\documentclass[english]{amsart}

\usepackage[all,cmtip]{xy}
\usepackage{todonotes}

\usepackage[titletoc]{appendix}

\usepackage{amsmath,amssymb,amscd,amsfonts}

\usepackage{babel}
\usepackage{amstext}
\usepackage{amsmath}
\usepackage{amsfonts}
\usepackage{latexsym}
\usepackage{ifthen}
\usepackage{xypic}
\usepackage{yfonts}
\usepackage[all,cmtip]{xy}
\xyoption{all}
\usepackage{enumitem}
\setlist[enumerate]{label=(\thethm.\arabic*), before={\setcounter{enumi}{\value{equation}}}, after={\setcounter{equation}{\value{enumi}}}}

\DeclareFontFamily{U}{MnSymbolC}{}
\DeclareSymbolFont{MnSyC}{U}{MnSymbolC}{m}{n}
\DeclareFontShape{U}{MnSymbolC}{m}{n}{
    <-6>  MnSymbolC5
   <6-7>  MnSymbolC6
   <7-8>  MnSymbolC7
   <8-9>  MnSymbolC8
   <9-10> MnSymbolC9
  <10-12> MnSymbolC10
  <12->   MnSymbolC12}{}
\DeclareMathSymbol{\intprod}{\mathbin}{MnSyC}{'270}

\newcommand{\bP}{\mathbb{P}}
\newcommand{\R}{\mathbb{R}}
\newcommand{\CC}{\mathbb{C}}

\newcommand{\ddbar}{\partial\bar{\partial}}

\newcommand{\wh}{\widehat}
\newcommand{\wt}{\widetilde}

\newcommand{\cF}{\mathcal{F}}

\newcommand{\cC}{\mathcal{C}}

\renewcommand{\O}{\mathcal{O}}

\newcommand{\ep}{\varepsilon}
\renewcommand{\epsilon}{\varepsilon}

\renewcommand{\ker}{\mathrm{Ker} \,}

\newcommand{\ol}{\overline}

\renewcommand{\leq}{\leqslant}
\renewcommand{\geq}{\geqslant}
\newcommand{\Ricci}{\mathrm{Ricci} \,}

\newcommand{\Vol}{\mathrm{Vol} \,}

\newcommand{\Tr}{\mathrm{Tr} \,}

\newcommand{\ddc}{dd^c}

\newcommand{\Supp}{\mathrm {Supp}}

\newcommand{\Ker}{\mathrm{Ker}}

\newcommand{\C}{\mathcal C}

\newcommand{\dbar}{\bar \partial}

\newtheorem{thm}{Theorem}[section]
\newtheorem{lemma}[thm]{Lemma}
\newtheorem{proposition}[thm]{Proposition}
\newtheorem{question}[thm]{Question}
\newtheorem{conjecture}[thm]{Conjecture}

\theoremstyle{remark}

\newtheorem{remark}[thm]{Remark}

\numberwithin{equation}{thm}

\title{On the Ohsawa-Takegoshi extension theorem}
\author{Junyan Cao, Mihai Paun\\ \\ {{--with an appendix by--}}\\ {{Bo Berndtsson}}}

\address{Department of Mathematics\\Chalmers University
  of Technology \\
 S-412 96
  G\"OTEBORG} 

\email{ bob@chalmers.se}

\address{Institut de Math\'ematiques de Jussieu\\ Universit\'e Paris 6\\ 
4 place Jussieu\\
Paris\\}

\email{junyan.cao@imj-prg.fr}

\address{Institut f\"ur Mathematik, Universit\"at Bayreuth\\ 
95440 Bayreuth}

\email{mihai.paun@uni-bayreuth.de}

\begin{document}

\maketitle

\tableofcontents
\section{Introduction}

\noindent Since it was established in \cite{OT}, the Ohsawa-Takegoshi extension theorem turned out to be a fundamental tool in complex geometry. As of today, there are
uncountable many proofs and refinements of the original result and even more applications to both complex analysis and algebraic geometry. Very roughly, the set-up is as follows: $u$ is a canonical form defined on a sub-variety $Y\subset X$ with values in a Hermitian bundle $F\to X$. 
We are interested in the following two main questions.
\begin{enumerate}

\item[$Q_1.$] Does the section $u$ extend to $X$?

\smallskip

\item[$Q_2.$] If the answer to the previous question is ``yes", can one construct an extension whose $L^2$ norm is bounded by the $L^2$ norm of $u$,
up to an universal constant?
\end{enumerate}

\noindent If $Y$ is non-singular, then the results in e.g. \cite{Man} give --practically optimal--
curvature conditions for the bundle $F$ such that the answer to both questions above is affirmative.
We refer to the articles \cite{BB1}, \cite{BB2}, \cite{BL16}, \cite{ChCh}, \cite{JP1}, \cite{JP3},
\cite{DHP}, \cite{GZ}, \cite{Mat}, \cite{MV1}, \cite{MV2}, \cite{MV3}, \cite{Oh1}, \cite{Oh2}, \cite{Oh3}, \cite{Oh4}, \cite{Oh5}, \cite{Oh6}, \cite{MP} for many interesting developments and applications.
\smallskip

\noindent The case of a singular sub-variety $Y$ turns out to be significantly more 
difficult
and the most complete results obtained so far only treat the \emph{qualitative} aspect of the extension problem, that is to say the question $Q_1$, cf. \cite{CDM}. 
\smallskip

\noindent
In this article we are concerned with the question $Q_2$. We obtain a few \emph{quantitative} results for extension of
twisted forms defined on sub-varieties $Y$ which have simple normal crossings. Our main motivation is the Conjecture in \cite{DHP}. To begin with we fix some notations/conventions.
\medskip

\noindent Let $X$ be a non-singular, projective manifold and let
$\displaystyle Y:= \sum_{i=1}^N Y_i$ be a divisor
with simple normal crossings. Let $(L, h_L)$ be a Hermitian line
bundle on $X$, endowed with a metric
$h_L$. The following assumptions will be in force throughout this article.
\begin{enumerate}

\item[(a)] The usual curvature requirements are satisfied
\begin{equation}\nonumber \Theta_{h_L}(L)\geq 0, \quad
  \Theta_{h_L}(L)\geq \delta \Theta_{h_Y} (Y) ,
\end{equation}
where $\delta> 0$ is a positive real number and $h_Y$ is a smooth metric on the bundle corresponding to $\mathcal O(Y)$.
Let $s$ be the canonical section of $\mathcal O(Y)$ with the normalisation condition $|s|_{h_Y} ^2 \leq e^{-\delta}$. 
\smallskip

\item[(b)] The singularities of the metric $h_L$ of $L$ are of the following type
\begin{equation}\nonumber \label{ot4}
\varphi_L= \sum_j r_j\log|f_j|^2+ \psi_L
\end{equation}
where $f_j$ are local holomorphic functions such that they are not identically
zero when restricted to any of the components of $Y$
and $k_j>0$ are positive integers. Moreover, we assume that 
 $\psi_L$ is bounded.
\smallskip

\item[(c)] Let $\displaystyle 
u\in H^0\big(X, (K_X+ Y+ L)\otimes \O_X/\O_X(-Y)\big)$ 
be a twisted canonical form defined over $Y$. 
There exists a covering $(\Omega_i)$ of $X$ with coordinate sets
such that the restriction of the section $\displaystyle u|_{\Omega_i}$ of $u$ admits an extension $U_i$ which belongs to the multiplier ideal of $h_L$, i.e.
\begin{equation} \label{ot6}
\int_{\Omega_i}|U_i|^2e^{-\varphi_L}< + \infty.
\end{equation}
  
\end{enumerate}  
\medskip

\noindent We note that near a non-singular point of $Y$ the existence of $U_i$ follows from the usual $L^2$ hypothesis of OT theorem provided that $u$ belongs to the
multiplier ideal sheaf of $h_L|_Y$. But this may no longer be true
in a neighborhood of singular point of $Y$.  
\medskip

\noindent In addition to the natural hypothesis (a), (b) and (c) above we collect next two other requirements we need to impose for some of our statements to hold. Let $V_{\rm sing}$ be an open subset of $X$ containing the singular locus of $Y$. 

\begin{enumerate}

\smallskip
  
\item[(i)] We assume that there exists an open subset $V_{\rm sing}$ of $X$ containing the singularities of $Y$ such that the following hold.

  \smallskip
\begin{itemize}  
  
\item[($\rm i.\alpha$)] There exists a snc divisor $W= \sum W_j$ on $X$ such that the singularities of the restriction of the metric $h_L$ of $L$ to $V_{\rm sing}$ are as follows
$$\varphi_L= \sum_j \left(1- \frac{1}{k_j}\right)\log|z_j|^2+ \tau_L$$  
where $k_j$ are positive integers, and $z_j$ are the local equations of $W$. The local weight $\tau_L$ above is assumed to be bounded, and smooth outside the support of $W$.
  
\smallskip
  
\item[(i.$\beta$)] The curvature of the restriction of $\displaystyle h_L|_{V_{\rm sing}}$ is greater than $\displaystyle C_{\rm sing}\omega_{\mathcal C}|_{V_{\rm sing}}$, where $C_{\rm sing}$ is a positive constant, and $\omega_{\mathcal C}$ is a fixed K\"ahler metric with conic singularities on $\big(X, \sum(1-1/k_i)W_i\big)$.
\end{itemize}

\medskip
  
\item[(ii)] There exists an open subset $V_{\rm sing}$ of $X$ containing the singularities of $Y$ such that the curvature of the restriction of $\displaystyle h_L|_{V_{\rm sing}}$ is identically zero.  
\end{enumerate}  
\medskip

\noindent In this context our first result states as follows.

\begin{thm}\label{otconique} We assume that the metric $\displaystyle h_L= e^{-\varphi_L}$ of $L$ and the section $u$ verifies the requirements {\rm{(a), (b)}, (c)}
  as well as {\rm (i)} above. Then $u$ extends to $X$, and for each
  $1\geq \gamma\geq 0$ there exists an extension $U$ of $u$ such that we have the estimates
\begin{align}\label{conic1}
 \frac{1}{C}\int_{X\setminus V_{\rm sing}} |U|^2e^{-\varphi_Y- \varphi_L}dV_{\omega_{\mathcal C}} \leq  &
  \int_{Y\setminus V_{\rm sing}}\left|\frac {u}{ds}\right|^2e^{-\varphi_L}\nonumber \\
+ & \sum_i\left(\int_{Y_i\cap V_{\rm sing}}
  \left|\frac{u}{ds}\right|^{\frac{2}{1+\gamma}}e^{-\frac{\varphi_L}{1+ \gamma}}dV_{\omega_{\mathcal C}}\right)^{1+\gamma}
\end{align}
where $\omega_{\mathcal C}$ is the reference metric on $X$ and the
constant $C$ depends on $\gamma$, the geometry of $(V_{\rm sing}, \omega_{\mathcal C})$, the positivity constant and the upper bound for $\displaystyle \Tr_{\omega_{\cC}}dd^c\tau$ in {\rm (i.$\beta$)}. 
\end{thm}

\begin{remark} The precise dependence of the constant $C$ in \eqref{conic1} of the quantities mentioned in Theorem \ref{otconique}
  can be easily extracted from the proof we present in Section 5. Moreover, one can easily construct a K\"ahler metric $\omega_{\cC}$ on $X$ such that it is smooth on $X\setminus \wt V_{\rm sing}$, and it has conic singularities along $\sum(1-1/k_i)W_i$
when restricted to $V_{\rm sing}$ only. Here $\wt V_{\rm sing}$ is any open subset of $X$ containing the closure of $V_{\rm sing}$.
  
\end{remark}

\begin{remark}\label{rk1}
  It is very likely that our arguments work under more general circumstances, e.g. one can probably establish the same result in the absence of the hypothesis (b) (via the regularisation procedure due to J.-P.~Demailly, cf. \cite{JP2}).
  But so far it is unclear to us how to
  remove the local strict positivity hypothesis in (i.$\beta$), or the fact that the singularities of $\displaystyle h_L|_{V_{\rm sing}}$ are assumed to be of conic type. 
\end{remark}
\medskip

\medskip

\noindent In conclusion, Theorem \ref{otconique} is providing an extension of $u$
whose
$L^2$ norm is bounded by the usual quantity outside the singularities of $Y$, and by
an \emph{ad-hoc} $L^{p}$ norm near $Y_{\rm sing}$, for any $p\in [1, 2[$. The example proving Claim 4 in the Appendix
shows that the this type of estimates are sharp.
\smallskip

However, the constant
$\displaystyle ``C"$ in \ref{otconique} involves the
geometry of the -local- pair $(V_{\rm sing}, \omega_{\mathcal C})$, or if one prefers,
the restriction of $h_L$ to $V_{\rm sing}$. Moreover,
we only allow singularities of conic type for $\displaystyle h_L|_{V_{\rm sing}}$. In order to try to ``guess'' the type of estimates one could hope for in general, we make the following observation. Let $\Omega\subset V_{\rm sing}$
be a coordinate open subset. The restriction of the RHS of \eqref{qi1} to
$\Omega$ is given by the following expression
\begin{equation}\label{obs1}
\int_{\Omega\cap Y}\frac{1}{\prod_j|f_j|^{\frac{2}{1+\gamma}}}\frac{|f_u|^{\frac{2}{1+\gamma}}d\lambda}{\prod_i |z_i|^{2(1-1/k_i)}}
\end{equation}
where $\prod f_j= 0$ is the local equation of $Y\cap \Omega$ and $\prod z_i= 0$ is the
equation of $W$. This can be rewritten as
\begin{equation}\label{obs2}
\int_{Y_i\cap \Omega}
  \left|\frac{u}{ds_Y}\right|^{\frac{2}{1+\gamma}}e^{-{\varphi_L}}d\lambda,
\end{equation}
so from this point of view the following important --and very challenging-- problem is natural.

\begin{conjecture}\label{cot2}
 We assume that the metric $\displaystyle h_L= e^{-\varphi_L}$ of $L$ and the section $u$ verifies the requirements {\rm{(a), (b)}, (c)}. Then $u$ extends to $X$, and for each $1\geq \gamma\geq 0$ there exists an extension $U$ of $u$ such that we have the estimates
\begin{align}\label{qi1}
 \frac{1}{C_\gamma(V_{\rm sing})}\int_{X\setminus V_{\rm sing}} |U|^2e^{-\varphi_Y- \varphi_L}dV_{\omega} \leq  &
  \int_{Y\setminus V_{\rm sing}}\left|\frac {u}{ds}\right|^2e^{-\varphi_L}\nonumber \\
+ & \sum_i\left(\int_{Y_i\cap V_{\rm sing}}
  \left|\frac{u}{ds}\right|^{\frac{2}{1+\gamma}}e^{-{\varphi_L}}dV_{\omega}\right)^{1+\gamma}
\end{align}
where $\omega$ is a reference K\"ahler metric on $X$ and $C_\gamma(V_{\rm sing})$
only depends on $(V_{\rm sing}, \omega)$ and the restriction of the metric $h_L|_Y$ to $Y$.
\end{conjecture}

\noindent It is our belief that the most subtle part of the previous conjecture would be to have an accurate estimate for the constant $C_\gamma(V_{\rm sing})$.

\medskip

\noindent Our next two results are of \emph{experimental} nature and therefore we
have decided to formulate them for surfaces only, so that we have $\dim(X)= 2$.
The same type of statements hold in arbitrary dimension, as one can easily convince oneself. The method of proof is completely identical to the case we explain here in detail, so for simplicity's sake we stick to the case of surfaces.
\smallskip

\noindent We fix next few more notations adapted to the pair
$(X, Y)$.

Let $\displaystyle (\Omega_i)_{i\in I}$ be covering of $X$ with open coordinate
subsets. By the simple normal crossing hypothesis we can choose co-ordinates
$\displaystyle z_i= (z_i^1, z_i^2)$ such
\begin{equation}\label{ot7}
  Y\cap \Omega_i= (z_i^1 \dots z_i^p= 0).
\end{equation}
for each $i\in I$ and some $p$ (depending on $i$). Let $(\theta_i)_{i\in I}$ be a
partition of unity subordinate to $\displaystyle (\Omega_i)_{i\in I}$.
\medskip

\noindent Since we assume that $X$ is a complex surface, the components of $Y$ are smooth curves. The singular set of $Y$ (i.e. the mutual intersections of its components) consists of a finite number of points of $X$, denoted by $p_1,\dots, p_s$. 

We assume that $\Omega_i$ is refined enough so that the section
$\displaystyle u|_{\Omega_i}$ is given by  
\begin{equation}\label{ot8}
f_idz_i^1\wedge dz_i^2
\end{equation}  
for some holomorphic function $f_i$. On overlapping open subsets, different expressions \eqref{ot8} are gluing only modulo a 2-form divisible by the equation of the divisor $Y$.

Let $p$ be one of the singular points of $Y$, assumed to be
the center of some $\Omega_i$. We denote by $t_i:= z_i^1\cdot z_i^2$; this is-- by our previous conventions-- the local equation of \emph{the cross} $Y\cap \Omega_i$. We can interpret the function (= $n-2$--form in general) $f_i$ as a local section of the bundle $\displaystyle L|_{\Omega_i}$, and as such we can consider
its derivative
\begin{equation}\label{ot9}
  \partial _{\varphi_L}f_i
\end{equation}
with respect to the Chern connection of $L$. The result is a $1$--form on $\Omega_i$.
\medskip

\noindent Given the hypothesis in our following statement, it is possible to construct an extension of 
$u$ by applying the result in \cite{Man}. However, here we obtain different type of estimates.

\begin{thm}\label{ot}
  Let $X$ be a smooth projective surface, and let $(L, h_L)$
  be a line bundle such that the usual curvature conditions {\rm (a)} and {\rm (b)} are satisfied. Assume moreover that $h_L$ is non-singular and
 for each $i\in I_{\rm sing}$ we have $f_i\in (z_i^1, z_i^2)$, in other words our section vanishes on the set singular points of $Y$.

\noindent Then there
  exists a holomorphic section $U$ of the bundle $K_X+ Y+ L$ for which the following hold.
\begin{enumerate}

\item[\rm(1)] The restriction $U|_{Y}$ of $U$ to $Y$ is equal to $u$. 
  \smallskip
  
\item[\rm(2)] There exists a constant $C(X, V_{\rm sing})> 0$ such that we have
\begin{align}\nonumber
 \frac{1}{C(X, V_{\rm sing})}\int_X|U|^2e^{-\varphi_Y-\varphi_L}\leq & 
  \int_{Y\setminus V_{\rm sing}}\log^2(\max|s_{j}|^2)\left|\frac {u}{ds}\right|^2e^{-\varphi_L}\nonumber \\
+ \label{log} & \int_{Y\cap V_{\rm sing}}\log^2(\max|s_{j}|^2)
  |\partial_{\varphi_L} u|^2e^{-\varphi_L}.
\end{align}
\end{enumerate}
\end{thm}

\medskip

\noindent We obtain the same type of result provided that the bundle
$(L, h_L)$ is flat near $Y_{\rm sing}$, as follows.

\begin{thm}\label{otflat}
  Let $X$ be a smooth projective surface, and let $(L, h_L)$
  be a line bundle such that the curvature and $L^2$ conditions {\rm (a)},
  {\rm (b)}, {\rm (c)} as well as the additional property {\rm (ii)} are satisfied.
  
\noindent Then there
  exists a holomorphic section $U$ of the bundle $K_X+ Y+ L$ enjoying the following properties.
\begin{enumerate}

\item[\rm(1)] The restriction $\displaystyle U|_{Y}$ is equal to $u$. 
  \smallskip
  
\item[\rm(2)] We have
\begin{align}\label{ot10}\nonumber
 \frac{1}{C(X, V_{\rm sing})}\int_{X\setminus V_{\rm sing}}|U|^2e^{-\varphi_Y-\varphi_L}\leq & 
  \int_{Y\setminus V_{\rm sing}}\log^2(\max|s_{j}|^2)\left|\frac {u}{ds}\right|^2e^{-\varphi_L}\nonumber \\
+ & \int_{Y\cap V_{\rm sing}}\log^2(\max|s_{j}|^2)
  |\partial_{\varphi_L} u|^2e^{-\varphi_L}.
\end{align}
\end{enumerate}
\end{thm}
\medskip

\noindent Our next statement is confined to the two-dimensional case.
\begin{thm}\label{otflat1}
Let $X$ be a smooth projective surface, and let $(L, h_L)$
be a line bundle such that the usual curvature and $L^2$ conditions {\rm (a)}, {\rm (b)} 
and {\rm (c)} are satisfied, respectively. We assume moreover that the following hold.
\begin{enumerate}
\smallskip

\item[\rm(1)] There exists a component $Y_1$ of $Y$ which intersects $\cup_{j\neq 1} Y_j$ 
in a unique point $p_1$, such that $u(p_1)\neq 0$.
\smallskip

\item[\rm(2)] The restriction $\displaystyle (L, h_L)|_{Y_1}$ is Hermitian flat.
\end{enumerate}  
Then the section $u$ admits an extension $U$ satisfying the same estimates as in Theorem \ref{ot}.
\end{thm}
\medskip

\noindent The \emph{raison d'\^etre} of the previous theorems
\ref{ot}, \ref{otflat} and \ref{otflat1} is that
the inequality \eqref{ot10} is meaningful even in the absence of the additional hypothesis 
these statements contain. Because of the variety of contexts in which
an extension of $u$ verifying the estimates of type (2) of Theorem \ref{otflat}
can be obtained,
it is very tempting to formulate the following.
\begin{conjecture}\label{cot}
  Let $(X, Y)$ be a smooth projective pair, where $X$ is a surface and
$Y$ is an snc divisor. Let $(L, h_L)$
  be a line bundle such that the properties {\rm (a)}, {\rm (b)} 
and {\rm (c)} are satisfied. Then there
  exists a holomorphic section $U$ of the bundle $K_X+ Y+ L$ enjoying the following properties.
\begin{enumerate}

\item[\rm(1)] The section $U$ is an extension of $u$. 
  \smallskip
  
\item[\rm(2)] We have
\begin{align}\label{ot11}\nonumber
 \frac{1}{C(X, V_{\rm sing})}\int_{X\setminus V_{\rm sing}}|U|^2e^{-\varphi_Y-\varphi_L}\leq & 
  \int_{Y\setminus V_{\rm sing}}\log^2(\max|s_{j}|^2)\left|\frac {u}{ds}\right|^2e^{-\varphi_L}\nonumber \\
+ & \lim_{\ep\to 0}\int_{Y\cap V_{\rm sing}}\log^2(\max|s_{j}|^2)
  |\partial_{\varphi_{L, \ep}} u|^2e^{-\varphi_{L, \ep}},
\end{align}
where $\displaystyle \varphi_{L, \ep}:= \log\left(\ep^2e^{\phi}+ e^{\varphi_L}\right)$
is a non-singular approximation of $h_L$.
\end{enumerate}
\end{conjecture}
\medskip

\noindent In the sequel of this article we will formulate the precise higher dimensional version of this conjecture, and we will explain its impact on the extension of the pluricanonical forms.

\noindent In the appendix A by Bo Berndtsson some examples are given that indicate that the estimates (2) in Conjecture \ref{cot}
are most likely the best one could hope for: without the $\log$ factor, this conjecture is simply wrong. Moreover, the example given in order to prove
Claim 4 shows that the factor $e^{-\varphi_L}$ in \eqref{qi1}
cannot be replaces by the slightly less singular weight $e^{-(1-\ep)\varphi_L}$,
for any $\ep> 0$. Finally, the appendix contains a comparison with a one-dimensional problem (the \emph{fat point}), intended to highlight the origin of the
difficulties in a very simple setting. 
\medskip

\subsection{Organization of the paper} In the second section we explain the
main ideas involved in the proof of our results. The next section is dedicated to the 
revision and slight improvement of the usual a-priori inequalities. Our principal contribution to the
\emph{Ohsawa-Takegoshi artisan industry} is in section four, where the necessary
tools from geometric analysis are recalled/developed. The proof of the results stated above is presented in section five.

\medskip

\section{An overview of the arguments} 

\noindent Our results are obtained by combining the method in \cite{OT} with the method in \cite{BB1}. In order to highlight the main arguments as well as some of the difficulties, we only discuss here the case of a
non-singular metric $h_L$ on $L$. In general the whole scheme of the proof becomes more technical, since the regularisation procedure we have to use for the metric is quite tricky to implement in the presence of the singular hypersurface $Y$. 

\noindent We start with a quick review of the usual case.

\subsubsection{The case of a non-singular hypersurface $Y$}
Let $\xi$ be a $L$-valued form of type $(n,1)$. We denote by
$\gamma_\xi:= \star \xi$ its Hodge dual (induced by an arbitrary K\"ahler form on $X$).

\noindent Consider the functional
\begin{equation}\label{eq130}
{\mathcal F}(\xi)=  \int_X\overline\partial
\left(\frac{u}{s_Y}\right)\wedge \overline{\gamma_\xi}e^{-\varphi_L}
\end{equation}
associated to the current $\displaystyle \overline\partial
\left(\frac{u}{s_Y}\right)$. 

We decompose $\xi= \xi_1+ \xi_2$ according to
$\Ker(\overline\partial)$ and $\Ker(\overline\partial)^\perp$. It turns out that  we have the equality
\begin{equation}\label{eq131}
{\mathcal F}(\xi)= \int_Y
\frac{u}{ds_Y}\wedge \overline{\gamma_{\xi_1}}e^{-\varphi_L}
\end{equation}
which is not completely obvious, given that the current defining $\cF$ is not in $L^2$.

\smallskip

\noindent We have $\displaystyle \frac{u}{ds_Y}\in L^2(e^{-\varphi_L})$, hence it is enough to find an upper bound for 
\begin{equation}\label{eq132} c_{n-1}\int_Y
  \gamma_{\xi_1}\wedge \overline{\gamma_{\xi_1}}e^{-\varphi_L}.
\end{equation}
This is done by the next estimate, which is derived in \cite{BB3} via the $\partial\overline\partial$--Bochner method due to Siu cf. \cite{Siu1}
\begin{equation}\label{eq133}
  c_{n-1}\int_Y\gamma_{\xi_1}\wedge \overline {\gamma_{\xi_1}}e^{-\varphi_L}\leq C\int_X\log^2(|s_Y|^2)|\overline\partial^\star{\xi_1}|^2e^{-\varphi_L}dV_\omega.
  \end{equation}
  \smallskip

  \noindent In conclusion we have
  
\begin{equation}\label{eq134}\left|\int_X\overline\partial
  \left(\frac{u}{s_Y}\right)\wedge \overline{\gamma_\xi}e^{-\varphi_L}\right|^2
\leq C\int_X\log^2(|s_Y|^2)|\overline\partial^\star{\xi}|^2e^{-\varphi_L}dV_\omega
\end{equation}
and the ``estimable'' extension will be obtained by using the solution of the equation
$\displaystyle \overline\partial
\left(\frac{u}{s_Y}\right)= \overline\partial v$. We define
$U:= s_Yv$ and then we have
\begin{equation}\label{eq135}
U|_{Y}= u\wedge ds_Y, \qquad
\int_X\frac{|U|^2}{|s|^2\log^2|s|^2}e^{-\varphi_Y-\varphi_L}\leq C
\int_Y|u|^2e^{-\varphi_L}.
\end{equation}
\medskip

\subsubsection{Difficulties in the case of an snc hypersurface $Y$} In our setting we have $Y= \bigcup Y_i$, and the difficulty steams from the fact that the functional
\begin{equation}\label{eq136}\int_Y
  \frac{u}{ds_Y}\wedge \overline{\gamma_{\xi_1}}e^{-\varphi_L}= \sum_i
\int_{Y_i}
\frac{u}{\prod_{j\neq i} s_j ds_{Y_i}}\wedge \overline{\gamma_{\xi_1}}e^{-\varphi_L}
\end{equation}
becomes a sum of expressions involving forms with log poles. We have
\begin{equation}\label{eq137} \frac{u}{\prod_{j\neq i} s_j ds_{Y_i}}\not\in L^2(e^{-\varphi_L}|_{Y_i})
\end{equation}  
in general, so the previous arguments are breaking down.
\smallskip

\noindent Nevertheless we do have
\begin{equation}\label{eq138}
\left|\frac{u}{ds_Y}\right|_{\omega}^{\frac{2}{1+\alpha}}\in L^1(Y, \omega|_{Y})
\end{equation}
near the singularities of $Y$ for any reasonable metric $\omega$. This means that we have to find an estimate of
the $L^\infty$ norm of $\displaystyle \gamma_{\xi_1}|_{V_{\rm sing}\cap Y}$ in terms of the RHS of
\eqref{eq133}.

To this end, we use a procedure due to Donaldson-Sun in \cite{DS}. This consists in the following simple observation.
Assume that the support of $\xi$ is contained in $X\setminus V_{\rm sing}$. Then we have
\begin{equation}\label{eq139}
\dbar \xi_1= 0,\qquad  \dbar^\star\xi_1|_{V_{\rm sing}}= 0
\end{equation}
in other words, the restriction of 
$\xi_1$ to $V_{\rm sing}$ is harmonic. As we learn from function theory,
harmonic functions satisfy the mean value inequality: this is what we implement in our context, and it leads to the proof of Theorem \ref{otconique}.
\smallskip

\noindent The drawback of this method is that in the end, the constant measuring the $L^2$ norm of the extension is far from being as universal as in the case $Y_{\rm sing}= \emptyset $. This is due to the fact that the quantity
$\displaystyle \Delta''|\xi_1|^2$ has a term \emph{with wrong sign} involving the trace of the
curvature of $(L, h_L)$ with respect to the metric $\omega$ on $X$. This trace
{is not bounded} e.g. if $h_L$ is singular and $\omega$ is a fixed, non-singular
K\"ahler metric. It is for this reason that the singularities of $h_L$ and
those of $\omega_{\mathcal C}$ must be the same in Theorem \ref{otconique}.  
\medskip

\section{A-priori inequalities revisited}

\noindent We first recall the following estimate, which is essentially due to \cite{BB3}.

\medskip

\begin{thm}\label{bobapprox} Let $(X, \omega)$ be a K\"ahler manifold, and $Y$ be simple normal crossing divisor in $X$.
Let $L$ be a line bundle on $X$ with a non-singular metric $h_L$ such that 
$$\Theta_{h_L} (L)\geq 0,\qquad  \Theta_{h_L} (L) \geq \delta \Theta_{h_Y} (Y) $$
for some $\delta >0$ small enough, where $h_Y$ is a smooth metric on $\mathcal{O}_X (Y)$ such that $\displaystyle |s_Y|^2_{h_Y}\leq e^{-\delta}$.
Let $\xi$ be a smooth $(n, 1)$ form with compact support and with values in $L$. We denote by $\gamma_\xi:= \star \xi$ the image of $\xi$ by the Hodge operator. Then 
we have
\begin{align}
  c_{n-1}\int_X&\frac{\tau^2}{(\tau^2+ |s_Y|^2)^2}\gamma_\xi\wedge \ol {\gamma_\xi}e^{-\varphi_L}\wedge
  \sqrt{-1}\partial s_Y\wedge\ol{\partial s_Y} \nonumber \\
\label{eq012}  \leq & C\int_X\log^2(\tau^2+ |s_{Y}|^2)\left(|\dbar^\star\xi|^2+ |\dbar\xi|^2\right)e^{-\varphi_L}dV_\omega \nonumber \\
\nonumber       
\end{align}
where $C$ is a numerical constant and $\tau$ is an arbitrary real number.
\end{thm} 
\medskip

\noindent Before giving the proof of Theorem \ref{bobapprox} we notice
that it implies the following statement.

\begin{thm}\label{BoB++} Let $(X, \omega)$ be a K\"ahler manifold, and $Y$ be simple normal crossing divisor in $X$.
Let $L$ be a line bundle on $X$ with a non-singular metric $h_L$ such that 
$$\Theta_{h_L} (L)\geq 0,\qquad  \Theta_{h_L} (L) \geq \delta \Theta_{h_Y} (Y) $$
for any $\delta >0$ small enough, where $h_Y$ is a smooth metric on $\mathcal{O}_X (Y)$.
Let $\xi$ be a smooth $(n, 1)$ form with compact support and with values in $L$. We denote by $\gamma_\xi:= \star \xi$ the image of $\xi$ by the Hodge operator. Then 
we have
\begin{equation}\label{eq12}
c_{n-1}\int_Y\gamma_\xi\wedge \ol {\gamma_\xi}e^{-\varphi_L}\leq C\int_X\log^2(|s_{Y}|^2)\left(|\dbar^\star\xi|^2+ |\dbar\xi|^2\right)e^{-\varphi_L}dV_\omega
\end{equation}
where $s_{Y}$ is the canonical section of $\O(Y)$, normalized in a way that works for the proof.
\end{thm} 
\medskip

\begin{proof}[Proof of Theorem {\rm \ref{bobapprox}}]  
  We note that this improves slightly the estimate of Bo Berndtsson in \cite{BB3}, but the proof is virtually the same. Nevertheless, we will provide a complete treatment for the convenience of the reader.
\medskip

\noindent To start with, we recall the following ``$\ddbar$-Bochner formula''.

\begin{lemma}[\cite{Siu0}]\label{Siu}
  Let $\xi$ be a $(n, 1)$--form with values in $(L, h_L)$ and compact support. We denote by $\gamma_\xi= \star \xi$ the Hodge $\star$ of $\xi$ with respect to a K\"ahler metric $\omega$. Let
\begin{equation}\label{eq13}
T_\xi:= c_{n-1}\gamma_\xi\wedge \ol \gamma_{\xi}e^{-\varphi}
\end{equation}
be the $(n-1, n-1)$--form on $X$ corresponding to $\xi$, where $c_{n-1}=
\sqrt{-1}^{(n-1)^2}$ is the usual constant. Then we have the equality
\begin{align}\label{eq14}
  \sqrt{-1}\ddbar T_\xi= & \left(-2\Re\langle \dbar \dbar^\star_\varphi\xi, \xi\rangle + \Vert \dbar \gamma_\xi\Vert^2+
                           \Vert\dbar^\star_\varphi\xi\Vert^2- \Vert \dbar \xi\Vert^2 \right)dV_\omega \\
  + & \Theta_{h_L}(L)\wedge T_\xi.\nonumber  
\end{align} 
\end{lemma}

\noindent We apply this in the following context. Consider the function $\displaystyle w:= \log\frac{1}{|s_Y|^2+ \tau^2}$. A quick computation gives
\begin{equation}\label{eq15}
\sqrt{-1}\ddbar w= \frac{|s_Y|^2}{|s_Y|^2+ \tau^2}\theta_Y- \frac{\tau^2}{(|s_Y|^2+ \tau^2)^2}\sqrt{-1}\partial s_Y\wedge \ol{\partial s_Y} 
\end{equation}
where $\displaystyle \theta_Y:= \Theta_{h_Y}\left(\O(Y)\right)$ is the curvature of the bundle $\O(Y)$ with respect to the metric $h_Y$.

\noindent We multiply the equality \eqref{eq14} with $w$ and integrate the resulting top form over $X$. The left hand side term is equal to the difference of two terms
\begin{equation}\label{eq16}
  c_{n-1}\int_X\frac{|s_Y|^2}{|s_Y|^2+ \tau^2}\theta_Y\wedge \gamma_\xi\wedge \ol \gamma_{\xi}e^{-\varphi_L}\end{equation}
and
\begin{equation}\label{eq116}
 c_{n-1}\int_X \frac{\tau^2}{(|s_Y|^2+ \tau^2)^2}\gamma_\xi\wedge \ol \gamma_{\xi}\wedge \sqrt{-1}\partial s_Y\wedge \ol{\partial s_Y}.
\end{equation}
and we see that \eqref{eq116} is the term we have to
estimate.

\noindent We drop the positive terms on the right hand side and we therefore get
\begin{align}\label{eq17}
  c_{n-1}\int_X \frac{\tau^2}{(|s_Y|^2+ \tau^2)^2}\gamma_\xi\wedge \ol \gamma_{\xi}\wedge & \sqrt{-1}\partial s_Y\wedge \ol{\partial s_Y}\leq 
                                                                                            \int_Xw\vert\dbar \xi|^2e^{-\varphi_L}dV_\omega \nonumber\\
+ &  2\Re \int_Xw\langle \partial_\varphi^\star\partial_\varphi \gamma_\xi, \gamma_\xi\rangle e^{-\varphi_L}dV_\omega\nonumber\\
- & c_{n-1}\int_X\left(w\Theta_{h_L}(L)- \frac{|s_Y|^2}{|s_Y|^2+ \tau^2}\theta_Y\right)\gamma_\xi\wedge \ol \gamma_{\xi}e^{-\varphi_L} \\
\nonumber
\end{align}

\noindent A first observation is that the curvature term \eqref{eq17} is negative, by the hypothesis of Theorem \eqref{bobapprox}. 
Moreover, by Stokes formula we have
\begin{equation}\label{eq18}
  \int_Xw\langle \partial_\varphi^\star\partial_\varphi \gamma_\xi, \gamma_\xi\rangle e^{-\varphi}dV_\omega= \int_Xw|\partial_\varphi \gamma_\xi|^2e^{-\varphi}dV_\omega+  \int_X\langle \partial_\varphi \gamma_\xi,
  \partial w\wedge \gamma_\xi\rangle e^{-\varphi}dV_\omega 
 \end{equation}
so we see that modulo the second term on the RHS of \eqref{eq18}, we are done.

\noindent In order to take care of it we use Cauchy-Schwarz inequality
and we obtain
\begin{align}\label{eq19}
\left|\int_X\langle \partial_\varphi \gamma_\xi,
  \partial w\wedge \gamma_\xi\rangle e^{-\varphi_L}dV_\omega \right|\leq &
\int_Xw^2|\partial_\varphi \gamma_\xi|^2e^{-\varphi_L}dV_\omega\\
+ &  c_{n-1}\int_X\gamma_\xi\wedge \ol \gamma_{\xi}\wedge \frac{\sqrt{-1}\partial w\wedge \dbar w}{w^2}e^{-\varphi_L}.
\nonumber  
\end{align}
Thus the new term to bound is
\begin{equation}\label{eq20}
c_{n-1}\int_X\gamma_\xi\wedge \ol \gamma_{\xi}\wedge \frac{\sqrt{-1}\partial w\wedge \dbar w}{w^2}e^{-\varphi_L}
\end{equation}
and as observed in \cite{BB3}, the quantity \eqref{eq20} is less singular that   
the LHS of \eqref{eq17}, which was our initial problem.
\smallskip

\noindent In order to obtain a bound for \eqref{eq20} we consider the
function $$\displaystyle w_1:= \log w.$$ We have
\begin{equation}\label{eq21}
  \sqrt{-1}\ddbar w_1= \frac{\sqrt{-1}\ddbar w}{w}-
  \frac{\sqrt{-1}\partial w\wedge \dbar w}{w^2}
\end{equation}
and we use the same procedure as before, but with $w_1$ instead of $w$.
The analogue of \eqref{eq16} and \eqref{eq116} read as
\begin{equation}\label{eq22}
  c_{n-1}\int_X\frac{\sqrt{-1}\ddbar w}{w}\wedge \gamma_\xi\wedge \ol \gamma_{\xi}e^{-\varphi_L}-
 c_{n-1}\int_X \gamma_\xi\wedge \ol \gamma_{\xi}\wedge \frac{\sqrt{-1}\partial w\wedge \dbar w}{w^2},
\end{equation}
and this is good, because the second term in \eqref{eq22} is the one we are now after.
We skip some intermediate steps because they are absolutely the same as in the preceding consideration, except that $w_1$ appears instead of $w$.  
After integration by parts, the new ``bad term'', i.e. the analog of the RHS of \eqref{eq18} in our current setting is
\begin{equation}\label{eq24}
\int_X\langle \partial_\varphi \gamma_\xi,
  \partial w_1\wedge \gamma_\xi\rangle e^{-\varphi_L}dV_\omega 
\end{equation}
for which we use Cauchy-Schwarz and the observation is that $\displaystyle \partial w_1\wedge \ol{\partial w_1}$
coincides with $\displaystyle \frac{\sqrt{-1}\partial w\wedge \dbar w}{w^2}$. 

\noindent As a result of this second part of the proof we infer that we have 
\begin{equation}\label{est100}
c_{n-1}\int_X \gamma_\xi\wedge \ol \gamma_{\xi}\wedge \frac{\sqrt{-1}\partial w\wedge \dbar w}{w^2}\leq 
C\int_X\log^2(|s_{Y}|^2+ \tau^2)\left(|\dbar^\star\xi|^2+ |\dbar\xi|^2\right)e^{-\varphi_L}dV_\omega
\end{equation}
Then Theorem \ref{bobapprox} follows,
by combining \eqref{est100} with \eqref{eq19}. 
\end{proof}
\medskip

\begin{remark}\label{estimates}
  Actually we can use the second part of the proof of Theorem \ref{BoB++} in order to get the estimates
\begin{equation}\label{eq30}
  c_{n-1}\int_X\gamma_\xi\wedge \ol {\gamma_\xi}e^{-\varphi_L}\wedge
  \frac{\partial \sigma\wedge \ol{\partial \sigma}e^{-\varphi_F}}{|\sigma|^2\log^2|\sigma|^2}
  \leq C\int_X\log^2(|\sigma|^2)\left(|\xi|^2+ |\dbar^\star\xi|^2+ |\dbar\xi|^2\right)e^{-\varphi_L}dV_\omega
\end{equation}
where $\sigma$ is a holomorphic section of a line bundle $(F, h_F)$ endowed with a non-singular metric $h_F$. The constant ``C'' in \eqref{eq30} depends on the
norm of the curvature of $(F, h_F)$.
Thus, we obtain an estimate of the norm of $\gamma_\xi$ \emph{in the tangential directions of} $\sigma=0$ with respect to the Poincar\'e-type measure associated to $\sigma$. If the curvature of $(L, h_L)$ is greater than the some (small) multiple of $\Theta_{h_F}(F)$, then we can remove the term $|\xi|^2$ in the formula \eqref{eq30}. 
\end{remark}
\medskip

\section{Geometric analysis methods and results}

\noindent Let $\xi$ be a $L$-valued form of $(n,1)$ type such that $\Supp(\xi)
\subset X\setminus (V_{\rm sing}\cup |H|)$. We recall that here $V_{\rm sing}$ is
an open subset of $X$ containing the singularities of $Y$, and $H$ is a hyperplane section containing the singularities of the metric $h_L$.

\noindent We consider the orthogonal decomposition
\begin{equation}\label{eq25}
\xi= \xi_1+ \xi_2
\end{equation}
where $\xi_1\in \Ker(\dbar)$ and $\xi_2\in \Ker(\dbar)^\perp$ with respect to
the fixed K\"ahler metric $\omega_{\cC}$ with conic singularities on $X$ and the given metric $h_L$ on $L$.

\medskip

\noindent $\bullet$  The convention during the current section is that
we denote by $``C"$ any constant which depends in an explicit way of the quantities we will indicate.

\subsection{Orthogonal decomposition, I: approximation}
\medskip

\noindent In the following sections we will use an approximation statement, for which the context is as follows.

We can write
\begin{equation}\label{orth15}
X\setminus H= \bigcup \Omega_m
\end{equation}
where each $\Omega_m$ is a Stein domain with smooth boundary. Let $\omega_m$ be a complete metric on $\Omega_m$. Corresponding to each positive $\delta$ we introduce
\begin{equation}\label{orth16}
\omega_{m, \delta}:= \omega_{\cC}+ \delta \omega_m;
\end{equation}
it is a complete metric on $\Omega_m$ such that $\omega_{m, \delta}> \omega_{\cC}$
and $\lim_{\delta\to 0}\omega_{m, \delta}= \omega_{\cC}$ for each $m$.

\noindent We remark that the $L^2$ norm of $\xi$ with respect to $\omega_{m, \delta}$ and $h_L|_{\Omega_m}$ is finite, given the pointwise
monotonicity of the norm of $(n,1)$--forms.
Then we can 
decompose the restriction of $\xi$ to each $\Omega_m$ as follows
\begin{equation}\label{orth17}
\xi|_{\Omega_m}= \xi_1^{(m,\delta)}+ \xi_2^{(m, \delta)}.
\end{equation}
\medskip

\noindent We establish next the following statement.

\begin{lemma}\label{lim2}
  We have
\begin{equation}\label{orth18}
\xi_1= \lim_{m,\delta}\xi_1^{(m,\delta)}
\end{equation}
uniformly on compact sets of $X\setminus H$.
\end{lemma}
\smallskip

\noindent The proof is based on the monotonicity of the $L^2$ norms
\begin{equation}\label{orth19}
\vert \rho\vert_{\omega_{m, \delta}}^2dV_{\omega_{m, \delta}}< \vert \rho\vert_{\omega}^2dV_{\omega}
\end{equation}
for each $m, \delta$ and for any form $\rho$ of type $(n,1)$ with values in $L$.
The details are as follows.

\begin{proof}
  Let $K\subset X\setminus (\varphi_L= -\infty)$ be a compact subset. In what follows we are using the notation $``\ep"$ to indicate the set of parameters $(m, \delta)$, and we assume that $m\gg 0$ so that $K\subset \Omega_m$.

\noindent  We first notice that for each parameter $\ep$ the form $\displaystyle \xi_1^{(\ep)}$ is smooth, and that it verifies the equation 
\begin{equation}\label{orth4}
\Delta_\ep''(\xi_1^{(\ep)})= \dbar\dbar^\star(\xi),
\end{equation}
where $\Delta_\ep''$ is the Laplace operator on $(n, 1)$-forms with values in
$(L, h_L)$ and $(\Omega_m, \omega_{m, \delta})$.
We also have
\begin{equation}\label{orth5}
\int_{\Omega_m}\left|\xi_1^{(\ep)}\right|^2_{\omega_\ep}e^{-\varphi_L}dV_{\omega_\ep}\leq \int_X\left|\xi\right|^2_{\omega_{\cC}}e^{-\varphi_L}dV_{\omega_{\cC}}= \Vert \xi\Vert^2
\end{equation}
given the fact that \eqref{orth18} is orthogonal. It follows that the
family
\begin{equation}\label{orth6}
\xi_1^{(\ep)}|_K
\end{equation}
is uniformly bounded in $\cC^\infty$ norm. We can therefore extract a limit
$\xi_1^{(0)}$ as $\ep\to 0$, uniform on compact subsets by the usual diagonal process. We remark that we have
\begin{equation}\label{orth7}
\dbar \xi_1^{(0)}= 0, \qquad \int_X\left|\xi_1^{(0)}\right|^2_{\omega_{\cC}}e^{-\varphi_L}dV_{\omega_{\cC}}< \infty
\end{equation}
given that each form $\displaystyle \xi_1^{(\ep)}$ is $\dbar$-closed, combined with \eqref{orth5}.
\smallskip

\noindent On the other hand, let $\rho$ be a $\dbar$-closed form of $(n,1)$--type with values in $L$. We assume moreover that $\rho$ is $L^2$ with respect to
$\omega_{\cC}$ and $h_L$. Then we equally have
\begin{equation}\label{orth8}
  \int_{\Omega_m}\left|\rho\right|^2_{\omega_\ep}e^{-\varphi_L}dV_{\omega_\ep}<
  \int_{\Omega_m}\left|\rho\right|^2_{\omega_{\cC}}e^{-\varphi_L}dV_{\omega_{\cC}}< \infty
\end{equation}
for each $\ep= (m, \delta)$, and open subset $\Omega_m$. We infer that
\begin{equation}\label{orth9}
\int_{\Omega_m}\langle \xi_2^{(\ep)}, \rho \rangle_{\omega_\ep} e^{-\varphi_\ep}dV_{\omega_\ep}= 0
\end{equation}
for each value of $m$ and $\ep$.

\noindent Let $(K_l)$ be an increasing exhaustion of $X\setminus (\varphi_L= -\infty)$ by relatively compact sets. If $m\gg 0$ (depending on $l$) then we have
\begin{equation}\label{orth10}
  \left|\int_{\Omega_m\setminus K_l}\langle \xi_2^{(\ep)}, \rho \rangle_{\omega_\ep}
    e^{-\varphi_L}dV_{\omega_\ep}\right|^2\leq C(\xi)\int_{X\setminus K_l}\left|\rho\right|^2_{\omega_\ep}e^{-\varphi_L}dV_{\omega_\ep}
\end{equation}
by Cauchy inequality combined with \eqref{orth8}.
It follows that 
\begin{equation}\label{orth9+}
\left|\int_{K_l}\langle \xi_2^{(\ep)}, \rho \rangle_{\omega_\ep} e^{-\varphi_\ep}dV_{\omega_\ep}\right|^2\leq C(\xi)\int_{X\setminus K_l}\left|\rho\right|^2_{\omega_{\cC}}e^{-\varphi_L}dV_{{\omega_{\cC}}}.
\end{equation}

\noindent By letting $\ep\to 0$ we infer that for each fixed $l$ we
have 
\begin{equation}\label{orth11}
  \left|\int_{K_l}\langle \xi_2^{(0)}, \rho \rangle_{{\omega_{\cC}}} e^{-\varphi_L}
    dV_{{\omega_{\cC}}}\right|^2\leq C(\xi)\int_{X\setminus K_l}\left|\rho\right|^2_{\omega_{\cC}}e^{-\varphi_L}dV_{{\omega_{\cC}}}.
\end{equation}
Next, the inequality \eqref{orth7} shows that $\xi_2^{(0)}$ is $L^2$-integrable
with respect to $(L, h_L)$ and $(X, \omega_{\cC})$. It follows that we have 

\begin{equation}\label{orth13}
  \int_{X\setminus K_l}\langle \xi_2^{(0)}, \rho \rangle_{\omega_{\cC}} e^{-\varphi_L}
  dV_{\omega_{\cC}}\to 0  
\end{equation}
as $l\to \infty$ since both $\rho$ and $\xi_2^{(0)}$ are in $L^2$.

\noindent In other words, the form $\xi_2^{(0)}$ is orthogonal to $\ker \dbar$
and since we have
\begin{equation}\label{orth14}
\xi= \xi_1^{(0)}+ \xi_2^{(0)}
\end{equation}  
our lemma is proved (thanks to the uniqueness of such decomposition).
\end{proof}  


\subsection{Orthogonal decomposition, II: mean value inequality}\label{ortzwei} We analyze
here the behavior of $\xi_1$ restricted to the set $V_{\rm sing}$. During the current subsection we make the following conventions.

\smallskip

\begin{enumerate}
  
\item[(i)] We work with respect to the K\"ahler metric $\omega_{\cC}$
  exclusively on $V_{\rm sing}\subset X$ (this will be understood even if we do not mention it explicitly) and with respect to the Hermitian metric $\displaystyle h_{L}$ defined in the previous section on $L$.

\item[(ii)] We denote by $\xi$ a $(n, 1)$ form with values in $L$ such that we have
  $$\Supp(\xi)
  \subset X\setminus (V_{\rm sing}\cup |H|).$$
  We use the notations in \eqref{eq25} for its orthogonal decomposition with respect to $(\omega_{\cC}, h_L)$.
\end{enumerate}

\noindent 
In this subsection we establish the next result.

\begin{thm}\label{meanV} We have the mean-value type inequality
\begin{equation}\label{eq62}
\sup_{\frac{1}{2}V_{\rm sing}} |\xi_{1}|^{2}e^{-\varphi_L}\leq C(V_{\rm sing})\int_{V_{\rm sing}}|\xi_{1}|^{2}e^{-\varphi_L}dV_{\omega_{\cC}} 
\end{equation}
where $C(V_{\rm sing})$ here is a constant which only depends on the allowed quantities i.e. the geometry of $(V_{\rm sing}, \omega_{\cC})$ as well as $\alpha$ and $\tau$ in the assumption $\rm (i.2)$.
\end{thm}
\noindent The norm of $\xi_1$ in \eqref{eq62} is measured with the conic metric
$\omega_{\cC}$.
\medskip

\noindent The proof of Theorem \ref{meanV} unfolds as follows 
(cf. \cite{DS}, \cite{GaGa} for similar computations). In order to simplify the notations, 
we drop the $e^{-\varphi_L}$ in \eqref{eq62}, and write $\displaystyle |\xi_1|^2$ 
to express the point-wise norm of $\xi_1$ with respect to $\omega_{\cC}$ and $h_L$. 
First we show that there exists a constant $C$ such that 

\begin{equation}\label{eq63}
\sup_{V_{\rm sing}\setminus |W|}|\xi_1|^2\leq C< \infty
\end{equation}
where we denote by $|W|$ the support of the divisor $W$. This is the main reason why we have to assume that the singularities of $h_L$ and $\omega$ are ``the same''
in Theorem \ref{otconique}.

\noindent After this, we establish a differential inequality
satisfied by the function $|\xi_1|^2$ \emph{in the complement of the set}
$$\left(\varphi_L= -\infty\right) \cap V_{\rm sing}.$$
This is standard, and it combines nicely with \eqref{eq63} and Moser iteration process to give \eqref{eq62}.

\medskip

\begin{proof}[Proof of Theorem {\rm \ref{meanV}}]

\smallskip First we establish the crucial boundedness result \eqref{eq63}. Let
$z_1,\dots, z_n$ be a set of local coordinates defined on a open subset $\Omega\subset V_{\rm sing}$. We assume that the $\displaystyle (z_i)_{i=1\dots n}$ are adapted to the pair $(X, W)$, meaning that the local equation of $\Omega\cap W$ is
$$z_1\dots z_p= 0.$$
for some $p\leq n$. By hypothesis, the weight of the metric $h_L$ can be
written as
\begin{equation}\label{rev1}
\varphi_L= \sum_{i=1}^p\Big(1-\frac{1}{k_i}\Big)\log|z_i|^2+ \tau,
\end{equation}
where $k_i$ are positive integers and $\tau$ is a bounded function defined on $\Omega$.

\noindent The restriction of $\xi_1$ to $\Omega$ has the following properties
\begin{equation}\label{rev2}
  \int_{\Omega}|\xi_1|^2dV_{\omega_{\cC}}< \infty,\qquad
\dbar \xi_1=0, \qquad \int_\Omega\langle\xi_1, \dbar \phi \rangle_{\omega_{\cC}}e^{-\varphi_L}dV_{\omega_{\cC}}= 0
\end{equation}
where $\phi$ is any smooth $(n, 0)$ form with compact support in $\Omega$ which is $L^2$-integrable, and such that $\dbar \phi$ is in $L^2$ as well.
\smallskip

\noindent These properties have a very neat interpretation in terms of ramified covers, as follows. Let
\begin{equation}\label{rev3}
\pi: \wh \Omega\to \Omega, \qquad \pi(w_1,\dots, w_n)= (w_1^{k_1},\dots,w_p^{k_p}, w_{p+1}, \dots, w_n)
\end{equation}
be the usual local covering map corresponding to the divisor $\sum _{i=1}^p\Big(1-\frac{1}{k_i}\Big) W_i$. We define the $(n, 1)$--form $\eta$ on $\wh \Omega$ as follows
\begin{equation}\label{rev4}
\eta:= \frac{1}{\prod_{j=1}^p w_j^{k_j-1}}\pi^\star \xi_1,
\end{equation}
and a first remark is that we have
\begin{equation}\label{rev5}
\int_{\wh \Omega}|\eta|^2_ge^{-\tau\circ \pi}dV_g< \infty,
\end{equation}
where $g$ is the inverse image of the conic metric $\displaystyle g:= \pi^\star \omega_{\cC}$. The relation \eqref{rev5} is an immediate consequence of the change of variables formula, combined with the expression of $\varphi_L$ in \eqref{rev1}.
It follows that
\begin{equation}\label{rev6}
\dbar \eta= 0
\end{equation}  
on $\wh \Omega$ (this is true point-wise outside the support of $W$, and it extends across $W$ by the arguments in \cite{JP2}).
\smallskip

\noindent Let $\alpha$ be a smooth $(n, 0)$--form on $\wh \Omega$ with compact support. We claim that we have
\begin{equation}\label{rev7}
\int_{\wh \Omega}\langle \eta, \dbar \alpha\rangle_g e^{-\tau\circ \pi}dV_g= 0.
\end{equation}
Indeed this is clear, given the equality
\begin{equation}\label{rev8}
  \int_{\wh \Omega}\langle \eta, \dbar \alpha\rangle_g e^{-\tau\circ \pi}dV_g=
\int_{\wh \Omega}\langle \pi^\star\xi_i, \dbar\wh \alpha\rangle_g e^{-\varphi_L\circ \pi}dV_g  
\end{equation}
where $\wh \alpha:= {\prod_{j=1}^p w_j^{k_j-1}}\alpha$. On the right-hand side of
\eqref{rev8} we can assume that $\wh \alpha$ is the inverse image of a
$(n, 0)$ form with compact support on $\Omega$. This is seen as follows: let $f$ be an element of the group $G$ acting on $\wh \Omega$. Then we have 
\begin{equation}\label{rev9}
\int_{\wh \Omega}\langle \pi^\star\xi_i, \dbar\wh \alpha\rangle_g e^{-\varphi_L\circ \pi}dV_g =
\int_{\wh \Omega}\langle \pi^\star\xi_i, \dbar (f^\star \wh \alpha) \rangle_g e^{-\varphi_L\circ \pi}dV_g 
\end{equation}
since all the other objects
involved in the integral in question are invariant by inverse image. It follows that 
\begin{equation}\label{rev10}
\int_{\wh \Omega}\langle \pi^\star\xi_i, \dbar\wh \alpha\rangle_g e^{-\varphi_L\circ \pi}dV_g =
\int_{\wh \Omega}\langle \pi^\star\xi_i, \dbar (\pi^\star \phi) \rangle_g e^{-\varphi_L\circ \pi}dV_g 
\end{equation}
where $\displaystyle \pi^\star \phi:= \frac{1}{|G|}\sum_{f\in G}f^\star \wh \alpha$.
Then our claim follows by the third property in \eqref{rev2}.
\medskip

\noindent In conclusion, the form $\eta$ is harmonic on $\wh \Omega$ with respect to the metric $g$ and the weight $e^{-\tau\circ \pi}$. It is in particular bounded,
and this is equivalent to \eqref{eq63}. \qed
  
\bigskip

\noindent We obtain next a lower bound for the Laplace of $|\xi_1|^2$. 
 To this end, we choose geodesic local coordinates $\displaystyle (z_i)_{i=1,\dots, n}$ for the K\"ahler metric
$\omega_{\cC}$ locally near a point $x_0\in V_{\rm sing}\setminus W$. Let $e_L$ be a local holomorphic frame of $L$, such that the induced weight
$\phi$ of $h_{L}$ verify the relations
\begin{equation}\label{eq79}
\phi(x_0)= 0, \qquad d\phi(x_0)= 0.
\end{equation}
We express $\xi_1$ locally with respect to these coordinates
\begin{equation}\label{eq80}
\xi_1= \sum\xi_{\ol \alpha}dz\wedge dz^{\ol \alpha}\otimes e_L
\end{equation}
where $dz:= dz^1\wedge \dots \wedge dz^n$. We then have
\begin{equation}\label{eq81}
|\xi_1|^2_{\omega_{\cC}}= \sum_{\alpha, \beta} \xi_{\ol \alpha}\ol{\xi_{\ol \beta}}g^{\ol \alpha \beta}\frac{e^{-\phi_\ep}}{\det g}.
\end{equation}
The formula for the Laplace operator is $\displaystyle \Delta''= \Tr_{\omega_{\cC}}\sqrt{-1}\ddbar$ and so we have  
\begin{align}
  \Delta''(|\xi_1|^2)= &\label{eq82} |\nabla \xi_1|_\ep^2+
2\sum_{\alpha, \beta} \Re\left(\xi_{\ol \alpha, p\ol q}g^{\ol q p}\ol{\xi_{\ol \beta}}g^{\ol \alpha \beta}\right)\frac{e^{-\phi}}{\det g} \\
  + &\label{eq83}  \sum_{\alpha, \beta} \xi_{\ol \alpha}\ol{\xi_{\ol \beta}}g^{\ol \alpha \beta}_{, p\ol q}g^{\ol q p}\frac{e^{-\phi}}{\det g} \\
  - &\label{eq84}  \sum_{\alpha, \beta} \xi_{\ol \alpha}\ol{\xi_{\ol \beta}}
 (\phi+ \log\det g)_{, p\ol q}g^{\ol q p}g^{\ol \alpha \beta}\frac{e^{-\phi}}{\det g} \\
\nonumber
\end{align}
where we denote by $\big(g_{p\ol q}\big)$ the coefficients of the metric
$\omega_{\cC}$ with respect to the geodesic coordinates $(z_i)$ and by $g$ the
corresponding matrix. 
\smallskip

\noindent In order to obtain an intrinsic expression of the terms containing the second derivative in the RHS of \eqref{eq82}, 
we recall that we have

\begin{equation}\label{eq85}
  \dbar^\star_{\varphi_L}\xi_1= (-1)^n\left(-\frac{\partial \xi_{\ol \alpha}}{\partial z_\beta}g^{\ol \alpha \beta}- \frac{\partial g^{\ol \alpha \beta}}{\partial z_\beta}\xi_{\ol \alpha}+
 \xi_{\ol \alpha}g^{\ol \alpha \beta} \frac{\partial \varphi_L}{\partial z_\beta}\right)dz\otimes e_L
\end{equation}
hence the next equality holds at $x_0$
\begin{equation}\label{eq86}
\langle \square_\ep \xi_1, \xi_1\rangle_\ep= \left(-\xi_{\ol \alpha, p\ol q}g^{\ol q \beta}\ol{\xi_{\ol \beta}}g^{\ol \alpha p}- \xi_{\ol \alpha}\ol{\xi_{\ol \beta}}g^{\ol \alpha \delta}_{, \delta \ol \gamma}g^{\ol\gamma \beta}+ \xi_{\ol \alpha}\ol{\xi_{\ol \beta}}\varphi_{L, \delta \ol \gamma}g^{\ol \alpha \delta}g^{\ol \gamma \beta}\right)\frac{e^{-\phi}}{\det g}
\end{equation}
where $\square_\ep:= [\dbar, \dbar^\star]$ is the Laplace operator acting on $L$-valued forms of $(n, 1)$ type. The formula \eqref{eq86} is only valid for closed forms, which is the case for
$\xi_1$. Also, we have $\displaystyle \xi_{\ol \alpha, p\ol q}= \xi_{\ol q, p\ol \alpha}$ and therefore (twice the real part of) the first term on the RHS of \eqref{eq86} coincides with the
second one on the RHS of  \eqref{eq82}.
\smallskip

\noindent Next, since the metric $\omega$ is K\"ahler we have $\displaystyle
g^{\ol \alpha \delta}_{, \delta \ol \gamma}= -g_{\alpha \ol\gamma, \delta\ol\delta}$ hence we obtain
\begin{equation}\label{eq87}
\xi_{\ol \alpha}\ol{\xi_{\ol \beta}}g^{\ol \alpha \delta}_{, \delta \ol \gamma}g^{\ol\gamma \beta}= \mathcal R _{\alpha\ol \beta}\xi_{\ol \alpha}\ol{\xi_{\ol \beta}} 
\end{equation}
where $\mathcal R _{\alpha\ol \beta}$ are the coefficients of the Ricci tensor of $\omega_{\cC}$.

\noindent The last term in \eqref{eq86} is simply $\displaystyle \theta_{\alpha \ol \beta}\xi_{\ol \alpha}\ol{\xi_{\ol \beta}}$ where $\displaystyle \theta_{\alpha \ol \beta}$ are the coefficients of $\Theta_{h_L}(L)$.
\smallskip

\noindent Again by the K\"ahler hypothesis the term \eqref{eq83} is equal to
\begin{equation}\label{eq88}
\mathcal R _{\alpha\ol \beta}\xi_{\ol \alpha}\ol{\xi_{\ol \beta}}
\end{equation}
and therefore we obtain
\begin{align}
  \Delta''(|\xi_1|^2)= &\label{eq89} |\nabla \xi_1|^2- 2\Re\langle \square \xi_1, \xi_1\rangle\\
  + &\label{eq90} 2\sum_{\alpha, \beta}\left(\theta_{\alpha \ol \beta}- \mathcal R _{\alpha\ol \beta}\right) \xi_{\ol \alpha}\ol{\xi_{\ol \beta}}+ \sum_{\alpha, \beta} \mathcal R _{\alpha\ol \beta} \xi_{\ol \alpha}\ol{\xi_{\ol \beta}}\\
- &\label{eq91} \sum_{\alpha, \beta}\left(\theta_{\alpha \ol \beta}- \mathcal R _{\alpha\ol \beta}\right)g^{\ol\beta \alpha}|\xi_1|^2\\ 
\nonumber
\end{align}
by collecting the previous equalities at $x_0$.
\smallskip

\noindent The Ricci curvature of the metric $\omega_{\cC}$ is uniformly bounded,
so the function
\begin{equation}\label{mos1}
  f^2:= |\xi_1|^2
\end{equation}
(where the norm is measured with respect to $\omega_{\cC}$ and $h_L$) 
verifies the following properties.
\begin{enumerate}
\smallskip

\item[(1)] We have $\sup_{V_{\rm sing}\setminus W}f< \infty$, and moreover $f$ is smooth on $V_{\rm sing}\setminus W$.
\smallskip

\item[(2)] The following differential inequality holds true
\begin{equation}\label{mos777}
\Delta''f^2\geq |\nabla f|^2- Cf^2
\end{equation}
where $C$ is a constant depending on the  Ricci curvature of the metric $\omega_{\cC}$ and the trace of $dd^c\tau$ with respect to it. 
\end{enumerate}

\noindent Indeed the inequality at the point (2) follows from \eqref{eq89},
since we have $$|\nabla \xi_1|^2= \big|\nabla |\xi_1|\big|^2.$$



\medskip

\noindent Based on (1) and (2) we can conclude in two ways: either show that Schoen-Yau 
mean-value inequality holds for functions $f$ which verify these properties (the proof would be a simple adaptation of the arguments presented in \cite{SY}), or use the Moser iteration procedure. In what follows, we use Moser procedure.

\bigskip



\noindent We show next that the following statement holds true.

\begin{lemma}\label{mean2} Let $f$ be the function defined in \eqref{mos1}. Then 
there exists a constant $C_1$ depending on  $C$ and $(V_{\rm sing}, \omega_{\cC})$ only such that we have 
\begin{equation}\label{mos3}
\sup_{\frac{1}{2}V_{\rm sing} \setminus W}f^2 \leq C_1\int_{V_{\rm sing}} f^2dV_{\omega_{\cC}}.
\end{equation}
\end{lemma}
\noindent Remark that the main point here is that the constant $C_1$ is \emph{independent of the $\sup$ norm in {\rm (1)}}. After establishing this statement we are basically done, i.e. this implies Theorem \ref{meanV} announced at the beginning of the  current section.
\smallskip

\begin{proof}[Proof of Lemma {\rm \ref{mean2}}]
 Let $\rho$ be a function which is equal to
  1 on $\displaystyle 1/2V_{\rm sing}$ and whose support is in $V_{\rm sing}$.
  Then we have
\begin{equation}\label{mos333}
|\Delta \rho|\leq C, \qquad |d\rho| \leq C
\end{equation}
where the norm of the differential in \eqref{mos333} is measured with respect to the reference metric $\omega$ on $X$. 

\noindent Following \cite{BB2}, there exists a family of functions
$\displaystyle (\Xi_\ep)_{\ep> 0}$ associated to the analytic subset
$\displaystyle W= (h_L= \infty)\cap V_{\rm sing}$ such that $\Supp (\Xi_\ep)\subset V_{\rm sing}\setminus W$ and
for each compact subset $K\subset X\setminus W$ we have $\Xi_\ep|_K= 1$ if
$\ep< \ep(K)$ is small enough.
Moreover we have 
\begin{equation}\label{mos4}
\int_X|d (\Xi_\ep)|^2dV_\omega\to 0, \qquad \int_X|\Delta (\Xi_\ep)|dV_\omega\to 0 
\end{equation}
as $\ep\to 0$. We recall very briefly the construction: let $\rho_\ep$ be a function 
which is equal to one on the interval $[0, \ep^{-1}]$ and which equals zero on 
$[1+ \ep^{-1}, \infty[$. Then we define 
\begin{equation}\label{rev20}
\Xi_\ep:= \rho_\ep\left(\log\big(\log\frac{1}{|s_W|^{2}}\big)\right)
\end{equation}
where $s_W$ is the sections whose zero set is $W$. Then with respect to the conic metric 
$\omega_{\cC}$ we have 
\begin{equation}\label{rev21} 
|d (\Xi_\ep)|^2_{\omega_{\cC}} \leq \frac{\rho_\ep'}{\log^2|s_W|^2}\sum \frac{1}{|z_j|^{2/k_j}},
\end{equation}
up to a constant, from which \eqref{mos4} follows (we get a similar inequality for the Laplacian of $\Xi_\ep$).

\noindent The existence of $\displaystyle (\Xi_\ep)_{\ep> 0}$ combined with the second property in \eqref{mos333} allows us to deal with the fact that $f$ is not necessarily smooth.
\medskip

\noindent The proof which follows is rather standard, but we will nevertheless provide a complete treatment for convenience. We refer to \cite{GT} for a general discussion about Moser iteration method, and more specifically to
\cite{GaGa} where this is implemented in a context which is very similar to
ours here. 
\smallskip

We multiply the differential inequality \eqref{mos777} with $\Xi_\ep \rho^2$ and integrate the result over $X$; we infer that we have
\begin{equation}\label{mos5}
\int_X \Xi_\ep\rho^2\Delta(f^2)dV_{\omega_{\cC}}\geq \int_X \Xi_\ep\rho^2|d(f)|^2dV_{\omega_{\cC}}.
\end{equation}
On the LHS of we integrate by parts. The terms containing derivatives of $\Xi_\ep$ are
\begin{equation}\label{rev22}
\int_X|\Delta (\Xi_\ep)|\rho^2 f^2dV_{\omega_{\cC}}, \int_X\langle d\Xi_\ep,  d(\rho^2)\rangle f^2dV_{\omega_{\cC}}
\end{equation}
and they tend to zero \emph{precisely} because of the uniform boundedness 
property (1) of $f$, together with \eqref{mos4}.
These terms are vanishing as $\ep\to 0$, and the inequality \eqref{mos5} becomes
\begin{equation}\label{mos6}
\int_X f^2\Delta(\rho^2)dV_{\omega_{\cC}}\geq \int_X \rho^2|d(f)|^2dV_{\omega_{\cC}}.
\end{equation}
On the other hand we write

\begin{equation}\label{mos7}
 \int_X \rho^2|\nabla f|^2dV_{\omega_{\cC}}\geq \frac{1}{2}\int_X \big|\nabla(\rho f)\big|^2dV_{\omega_{\cC}}
- \int_X f^2|\nabla \rho|^2 dV_{\ep}, 
\end{equation}

\noindent which combined with \eqref{mos6} gives
\begin{equation}\label{mos8}
\int_X \big|\nabla(\rho f)\big|^2dV_{\omega_{\cC}}\leq C\int_{V_{\rm sing}} f^2 dV_{\omega_{\cC}}.
\end{equation}
where the constant $C$ in \eqref{mos8} only depends on the norm of the derivative of $\rho$.
\medskip

\noindent The following version of Sobolev inequality is
a direct consequence of \cite{Li1}, page 153.

\begin{thm}
  There exists a constant $C>0 $ such that the following holds
\begin{equation}\label{eq106}
 \frac{1}{C} \left(\int_X|f|^{\frac{2n}{n-1}}dV_{\omega_{\cC}}\right)^{\frac{n-1}{n}}\leq
\int_X|f|^2dV_{\omega_{\cC}}+ \int_X|\nabla f|^2dV_{\omega_{\cC}}   
\end{equation}
for any function $f$ on $X$.
\end{thm}

\noindent We therefore infer that we have
\begin{equation}\label{mos9}
  \left(\int_{1/2V_{\rm sing}}|f|^{\frac{2n}{n-1}}dV_{\omega_{\cC}}\right)^{\frac{n-1}{n}}
\leq C \int_{V_{\rm sing}} |f|^2 dV_{{\omega_{\cC}}}.  
\end{equation}
\smallskip

\noindent In order to obtain estimates for higher norms, we use \eqref{eq106} for
$\displaystyle f:= \Xi_\ep\rho f^{\frac{p}{2}}$ and we obtain
\begin{equation}\label{mos10}
\frac{1}{C} \left(\int_X(\Xi_\ep\rho)^{\frac{2n}{n-1}}f^{\frac{pn}{n-1}}dV_{\omega_{\cC}}\right)^{\frac{n-1}{n}}\leq
\int_X(\Xi_\ep\rho)^2f^pdV_{\omega_{\cC}}+ \int_X\left|\nabla\left(\Xi_\ep\rho f^{\frac{p}{2}}\right)\right|^2dV_{\omega_{\cC}}. 
\end{equation}
We show now that the second term of the right hand side of \eqref{mos10} verifies the inequality
\begin{equation}\label{mos11}
\int_X\left|\nabla\left(\rho f^{\frac{p}{2}}\right)\right|^2dV_{\omega_{\cC}}\leq Cp\int_X(\rho^2+ |\nabla \rho|^2)f^p dV_{\omega_{\cC}}.
\end{equation}
This is done using integration by parts: we have
$$\displaystyle \nabla\left(\rho \Xi_\ep f^{\frac{p}{2}}\right)= f^{\frac{p}{2}}\nabla (\rho\Xi_\ep)+ \frac{p}{2}\Xi_\ep\rho f^{\frac{p-2}{2}}\nabla f,$$ so we have to obtain
a bound for the term
\begin{equation}\label{mos12}
  \int_X(\rho\Xi_\ep)^2f^{{p-2}}|\nabla f|^2dV_{\omega_{\cC}}=
  \frac{1}{2}\int_X(\rho\Xi_\ep)^2f^{{p-3}}\langle \nabla f^2, \nabla f|\rangle dV_{\omega_{\cC}}. 
\end{equation}
We have
\begin{align}
  (p-2)& \int_X(\rho\Xi_\ep)^2f^{{p-3}}\langle \nabla f^2, \nabla f\rangle dV_{\omega_{\cC}}  \nonumber \\
  = & - \int_X(\rho\Xi_\ep)^2f^{{p-2}}\Delta f^2dV_{\omega_{\cC}} - 2\int_Xf^{{p-2}}\langle (\rho\Xi_\ep)\nabla f,
      f\nabla (\rho\Xi_\ep)\rangle dV_{\omega_{\cC}} \nonumber\\
 \leq  \label{mos13} & -\int_X(\rho\Xi_\ep)^2f^{{p-2}}|\nabla(f)|^2dV_{\omega_{\cC}} -
         2\int_Xf^{{p-2}}\langle (\rho\Xi_\ep)\nabla f, f\nabla (\rho\Xi_\ep)\rangle dV_{\omega_{\cC}} \\
 + &  C\int_X(\rho\Xi_\ep)^2f^{{p}}dV_{\omega_{\cC}}  \nonumber \\    
  \leq & C\int_X((\rho\Xi_\ep)^2+ |\nabla (\rho\Xi_\ep)|^2)f^pdV_{\omega_{\cC}}
\nonumber  
\end{align}
and as before, the terms involving $\nabla (\Xi_\ep)$ tend to zero as $\ep\to 0$.
We therefore get the inequality \eqref{mos11}. Remark that we are using the inequality \eqref{mos777} in order to obtain \eqref{mos13}.

\smallskip

\noindent We define $V_i:= (1/2+ 1/2^i)V_{\rm sing}$ and let $\rho_i$ be a cutoff function such that $\rho_i= 1$ on $V_{i+1}$ and such that $\Supp(\rho_i)\subset V_i$. Then we have $|\nabla \rho_i|\leq C2^i$, and by using \eqref{mos11} combined with the usual iteration process, Lemma \ref{mean2} follows.
\end{proof}
\medskip

\noindent Theorem \ref{meanV} is proved.\end{proof}
\medskip

\begin{remark}
  Actually a careful examination of the proof shows that one can obtain a
  constant $``C"$ in Lemma \ref{mean2} as follows
  \begin{equation}\label{mean100}
C= \frac{C(X, \omega)}{\Vol(V_{\rm sing})}.
\end{equation}
If necessary, this can be obtained by adapting the arguments of Schoen-Yau
in \cite{SY}, page 75.
\end{remark}  

\section{Proof of the main results}

\subsection{Proof of Theorem \ref{otconique}} We consider the ``usual'' family of cut-off functions
\begin{equation}\label{fin1}
\rho_\ep: X\to \R,\qquad \rho_\ep(z)= \rho\left(\frac{|s_Y|^2}{\ep^2}\right)
\end{equation}
where $\rho$ is a function defined on the set of positive 
real numbers such that $\rho= 1$ on $[0, 1]$ and $\rho= 0$ on $[2, \infty[$.

\noindent We will show here that the following a-priori inequality holds
\begin{equation}\label{eq111}
  \left|\int_X\dbar \left(\rho_\ep U_0\right)\wedge \ol{\gamma_{\xi}}e^{-\phi_L}\right|^2
  \leq C_\ep(U_0) \int_X\prod\log^2(|s_{Y_j}|^2+ \ep^{2/N})|\dbar^\star\xi|^2e^{-\phi_L}dV_{\omega_\mathcal{C}}
\end{equation}
where $N$ is the number of components of $Y$ and $\xi$ is a $(n, 1)$-form with values in $Y+ L$ whose support is contained in the complement of $V_{\rm sing}\cup H$. Also, $U_0$ 
is an arbitrary holomorphic extension of $u$, cf. \cite{CDM}, and we denote by $\phi_L$ the metric
\begin{equation}\label{fin100}
\phi_L:= \varphi_L+ \log|f_Y|^2
\end{equation}
on the bundle $L+ Y$.
We will see that the constant 
$C_\ep(U_0)$ in \eqref{eq111} is explicit, and it converges to the RHS of 
\eqref{qi1} as $\ep\to 0$. 
 Note that  all the
integrals above are at least well defined, given the condition we impose on the
support of $\xi$.
\medskip

\noindent The proof of \eqref{eq111} will be presented along the following line of arguments.
\smallskip

\noindent $\bullet$ Consider a $(n,1)$-form $\xi$ as above together with the orthogonal decomposition $\displaystyle \frac{\xi}{s_Y}= \xi_1+ \xi_2$ we have already discussed in detail in the previous section. Then we show that we have 
\begin{equation}\label{eq65}
  c_{n}\int_X\frac{\ep^2}{(\ep^{2}+ |s_Y|^2)^2}\gamma_{\xi_1}\wedge \partial s_Y\wedge \ol{\gamma_{\xi_1}\wedge \partial s_Y}e^{-\varphi_L}
  \leq 
\int_X\prod\log^2(|s_{Y_j}|^2+ \ep^{2/N})|\dbar^\star\xi|^2e^{-\phi_L}dV_{\omega_\mathcal{C}}.
\end{equation}
up to a numerical constant. 
This will be done by an approximation argument, using Lemma \ref{lim2}
as well as Theorem \ref{bobapprox}.
\smallskip

\noindent $\bullet$ The norm of the functional on the LHS of \eqref{eq111}
is evaluate in two steps
on the set $V_{\rm sing}$ we use Theorem \ref{meanV}, combined with a few simple
calculations. In the complement $\displaystyle X\setminus V_{\rm sing}$ the arguments are rather standard: we will use \eqref{eq65}.
\medskip

\noindent The remaining part of the current section is organized as follows.
We first show that \eqref{eq111} implies the existence of an ``estimable extension'' 
of $u$. Then we prove that the estimate \eqref{eq111} holds true.
  
\subsubsection{Functional analysis} Our method relies  on the next statement.
\begin{thm}\label{boestim} 
Let $u\in H^0 (Y, (K_X +Y+ L) |_Y)$ be a holomorphic section. 
We assume that there is a constant $C_\ep(U)$ such that for any $L$-valued smooth test form $\xi$ of type $(n, 1)$ with compact support in $X\setminus (V_{\rm sing}\cup|W|)$ the a-priori inequality \eqref{eq111} holds.
Then $u$ admits an extension $U\in H^0 (X, K_X+Y+ L)$ such that
\begin{equation}\label{bt702}
\int_{X\setminus V_{\rm sing}}\frac{|U|^2}{|s_Y|^{2}\prod\log^2(|s_{Y_j}|^2)}e^{-\varphi_L- \varphi_Y}dV_\omega\leq \lim_{\ep\to 0}C_\ep(U).
\end{equation}
\end{thm}
\begin{proof}
  This is done as in the classical case, by considering the vector subspace
\begin{equation}\label{bt703}
E:= \big\{\dbar^\star\xi : \xi \in C^2_c\big(X\setminus (V_{\rm sing}\cup H)\big)\big\}
\end{equation}
 of the $L^2_{n,0}(X, Y+L)$ forms, endowed with the scalar product induced by 
\begin{equation}\label{bt830}
\Vert \rho\Vert^2:= \int_X|\rho|^2 e^{-\phi_L}\prod\log^2(|s_{Y_j}|^2+ \ep^{2/N})dV_{\omega}.
\end{equation}
The functional
\begin{equation}\label{bt704}
  \dbar^\star\xi\to \int_X\dbar \left(\rho_\ep U_0\right)\wedge \ol{\gamma_{\xi}}e^{-\phi_L}
\end{equation}
is well-defined and bounded on $E$ by \eqref{eq111}, hence it extends by Hahn-Banach. The representation theorem of Riesz implies that there exists some form $$\displaystyle v\in L^2_{n,0}(X, Y+ L)$$ such that we have
\begin{equation}\label{bt705}
  \int_X\left\langle \dbar \left(\rho_\ep U_0\right), {\xi}\right\rangle e^{-\phi_L}=
  \int_X \left\langle v, \dbar^\star \xi \right\rangle e^{-\phi_L}\prod\log^2(|s_{Y_j}|^2+ \ep^{2/N})dV_{\omega_L}
\end{equation}
for all test forms $\xi\in E$ and such that 
\begin{equation}\label{bt706}
  \int_X|v|^2e^{-\phi_L}\prod\log^2(|s_{Y_j}|^2+ \ep^{2/N})dV_{\omega_L}\leq C_\ep(U_0).
\end{equation}
The equation \eqref{bt705} shows that we have
\begin{equation}\label{bt831}
\dbar\left(\rho_\ep U_0\right)= \dbar\left(\prod\log^2(|s_{Y_j}|^2+ \ep^{2/N})v\right)
\end{equation}
on $X\setminus V_{\rm sing}$. On the other hand, the form
\begin{equation}\label{bt832}
\rho_\ep U_0- \left(\prod\log^2(|s_{Y_j}|^2+ \ep^{2/N})\right)v
\end{equation}
is in $L^2(X\setminus V_{\rm sing})$: this is clear for the first term, as for the second one it is a consequence of \eqref{bt706}. 
\medskip

\noindent We infer that the form
\begin{equation}\label{bt833}
U_\ep:= \rho_\ep U_0- \left(\prod\log^2(|s_{Y_j}|^2+ \ep^{2/N})\right)v
\end{equation}
extends holomorphically on $X\setminus V_{\rm sing}$. This implies that
$\displaystyle v|_{X\setminus V_{\rm sing}}$ is non-singular, in particular $v$
equal zero when restricted to $Y\setminus V_{\rm sing}$ --given the estimates
in \eqref{bt706}.

\noindent Therefore we infer the equality

\begin{equation}\label{bt835}
U_\ep |_{Y\setminus V_{\rm sing}}= u.
\end{equation} 
We remark that $U_\ep$ extends to $X$ by theorem of Hartog's. This is clear if
$X$ is a surface cf. e.g. \cite{Hor}, Theorem 2.3.2. The general case follows as a
consequence of this, by a simple argument of slicing which we will not detail here.

\noindent Finally, the estimate for the $L^2$
norm of $U$ is deduced from \eqref{bt706}: we have
\begin{equation}\label{bt900}
\int_{X\setminus V_{\rm sing}}\frac{|U_\ep|^2}{|s_Y|^2\prod\log^2(|s_{Y_j}|^2+ \ep^{2/N})}e^{-\varphi_L- \varphi_Y}
\leq C_\ep(U_0)
\end{equation}
modulo a quantity which tends to zero as $\ep \to 0$. The conclusion follows.
\end{proof}

\subsubsection{End of the proof} We prove now the inequality \eqref{eq111}. As we have already mentioned, one of the main part of the proof is based on the a-priori estimate \eqref{eq65} which we derive here from
Theorem \ref{BoB++} combined with the results established in the first part of
section 3.
\medskip

\noindent We start with the following technical result, which plays a key role in the arguments to come. In order to simplify the notations, we write $\xi$ instead of the quotient $\displaystyle \frac{1}{s_Y}\xi$. 
\begin{proposition}\label{unifor1} Consider the orthogonal decomposition $\xi= \xi_1+ \xi_2$. Then the following hold: for each positive $\ep$ we have
\begin{align}
  \sum_i c_{n-1}\int_X&\frac{\ep^2}{(\ep^2+ |s_Y|^2)^2}\gamma_{\xi_1}\wedge \ol {\gamma_{\xi_1}}e^{-\varphi_L}\wedge
  \sqrt{-1}\partial s_i\wedge\ol{\partial s_i} \nonumber \\
  \leq & \label{end30}C\int_X\prod\log^2(\ep^{2/N}+ |s_{j}|^2)\left(|\dbar^\star\xi|^2\right)e^{-\varphi_L}dV_{\omega_\mathcal{C}} \\     
\nonumber\end{align}  
where $N$ is the number of components of $Y$.
\end{proposition}

\noindent Remark that we have the equality $\displaystyle |\dbar^\star\mu|^2e^{-\varphi_L}= |\dbar^\star\xi|^2e^{-\phi_L}$ if $\displaystyle \mu= \frac{1}{s_Y}\xi$, so the estimate \eqref{end30} is precisely what we have to prove.

\begin{proof} We recall the context in section 3: we have considered
an exhaustion 
\begin{equation}\label{end31}
X\setminus H= \bigcup \Omega_m
\end{equation}
where each $\Omega_m$ was a Stein domain with smooth boundary,
endowed with the family of complete metrics $\omega_{m, \delta}$ cf.
\eqref{orth16}. The restriction of $\xi$ to each $\Omega_m$ decomposes as follows
\begin{equation}\label{end32}
\xi|_{\Omega_m}= \xi_1^{(m,\delta)}+ \xi_2^{(m, \delta)}
\end{equation}
according to $(\Omega_m, \omega_{m, \delta})$ and $(L, h_L)$. 

\noindent We apply the inequality in Theorem \ref{bobapprox} for
$\xi_1^{(m,\delta)}$ and we get
\begin{align}
  \sum_i c_{n-1}\int_{\Omega_m}&\frac{\ep^2}{(\ep^2+ |s_Y|^2)^2}
                                 \gamma_{\xi_1^{(m,\delta)}}\wedge \ol {\gamma_{\xi^{(\ep)}_1}}e^{-\varphi_L}\wedge
  \sqrt{-1}\partial s_Y\wedge\ol{\partial s_Y} \nonumber \\
  \leq & \label{fin5} C\int_{\Omega_m}\prod\log^2(\ep^{2/N}+ |s_{j}|^2)\left|\dbar^\star(\xi_1^{(m,\delta)})\right|^2e^{-\varphi_{L}}dV_{\omega_{m, \delta}}. \\       
\nonumber\end{align}

\noindent Indeed we can use Theorem \ref{bobapprox} in this context even if
the form does not have compact support because 
the metric $\omega_{m, \delta}$ is complete. This has another consequence: we have the equality $\dbar^\star(\xi_1^{(m,\delta)})= \dbar^\star(\xi)$.
For the inequality \eqref{fin5} we have used the inequality
$$\log^2(\ep^2+ |s_Y|^2)\leq C \prod\log^2(\ep^{2/N}+ |s_{j}|^2)$$
where $N$ is the number of components of $Y$.

Let $K\subset X$ be any open set with compact closure in $X\setminus H$; for any $m\geq m_0(K)$ we have $\ol K\subset \Omega_m$ 
so the inequality \eqref{fin5} implies
\begin{align}
  \sum_i c_{n-1}\int_{K}&\frac{\ep^2}{(\ep^2+ |s_Y|^2)^2}
                                 \gamma_{\xi_1^{(m,\delta)}}\wedge \ol {\gamma_{\xi_1^{(m,\delta)}}}e^{-\varphi_L}\wedge
  \sqrt{-1}\partial s_Y\wedge\ol{\partial s_Y} \nonumber \\
  \leq & \label{fin6} C\int_{\Omega_m}\prod\log^2(\ep^{2/N}+ |s_{j}|^2)\left|\dbar^\star\xi\right|_{\omega_{m, \delta}}^2e^{-\varphi_{L}}dV_{\omega_{m, \delta}}. \\       
\nonumber\end{align}
Now the support of $\xi$ is a compact contained in $X\setminus H$, so if $m$ is large enough the boundary of $\Omega_m$ is disjoint from $\Supp(\xi)$.
A limit process (i.e. $\delta\to 0, m\to \infty$), together with Lemma \ref{lim2} implies that we have 
\begin{align}
  \sum_i c_{n-1}\int_{K}&\frac{\ep^2}{(\ep^2+ |s_Y|^2)^2}
                                 \gamma_{\xi_1}\wedge \ol {\gamma_{\xi_1}}e^{-\varphi_L}\wedge
  \sqrt{-1}\partial s_Y\wedge\ol{\partial s_Y} \nonumber \\
  \leq & \label{fin7} C\int_X \prod\log^2(\ep^{2/N}+ |s_{j}|^2)\left|\dbar^\star\xi\right|_{\omega_{\cC}}^2e^{-\varphi_{L}}dV_{\omega_{\cC}}. \\       
\nonumber\end{align}

\noindent The compact subset $K$ in \eqref{fin7} is arbitrary, 
so Proposition \eqref{unifor1} is proved.
\end{proof}

\medskip

\noindent We are now ready to finish the proof of Theorem \ref{otconique}.
Consider the integral
\begin{equation}\label{eq118}
\int_X\left \langle \dbar \left(\rho_\ep U_0\right), {\xi}\right\rangle e^{-\phi_L}
dV_{\omega_{\cC}}
\end{equation}
which up to a sign equals 
\begin{equation}\label{fin8}
\int_X\rho'\left(\frac{|s_Y|^2}{\ep^2}\right)
\left\langle U_0\wedge \frac{\ol{\partial s_Y}}{\ep^2}, {\frac{\xi}{s_Y}}\right\rangle e^{-\varphi_L}
dV_{\omega_{\cC}}. 
\end{equation}
We decompose as usual $\displaystyle \frac{\xi}{s_Y}= \xi_1+ \xi_2$ and then \eqref{fin8}
becomes
\begin{equation}\label{fin9}
\int_X\rho'\left(\frac{|s_Y|^2}{\ep^2}\right)
\left\langle U_0\wedge \frac{\ol{\partial s_Y}}{\ep^2}, \xi_1\right\rangle e^{-\varphi_L}
dV_{\omega_{\cC}}= \int_X\rho'\left(\frac{|s_Y|^2}{\ep^2}\right)
U_0\wedge \ol{\partial s_Y\wedge \gamma_{\xi_1} }\frac{e^{-\varphi_L}}{\ep^2}
\end{equation} 

\noindent We split its evaluation into two parts. The first one is
\begin{equation}\label{eq119}
  \int_{X\setminus V_{\rm sing}}\rho'\left(\frac{|s_Y|^2}{\ep^2}\right)
U_0\wedge \ol{\partial s_Y\wedge \gamma_{\xi_1} }\frac{e^{-\varphi_L}}{\ep^2}
\end{equation}
and by Cauchy-Schwarz inequality the square of its absolute value is smaller than
\begin{equation}\label{eq120}
\int_{K_\ep} |U_0|^2\frac{e^{-\varphi_L}}{\ep^2}dV_{\omega_{\cC}}\cdot  \int_{K_{\ep}}
|\partial s_Y\wedge \gamma_{\xi_1}|^2_{\omega_{\cC}}\frac{e^{-\varphi_L}}{\ep^2}dV_{\omega_{\cC}}
\end{equation}
where $K_\ep$ is the support of the function $\rho'\left(\frac{|s_Y|^2}{\ep^2}\right)$. We remark that we have 
\begin{equation}\label{fin10}
\frac{1}{\ep^2}\simeq \frac{\ep^2}{(\ep^2+ |s_Y|^2)^2}
\end{equation}
on the set $K_\ep$. Therefore, the second factor of the product \eqref{eq120}
is smaller than 
\begin{equation}\label{fin11}
C\int_X\prod\log^2(\ep^2+ |s_{j}|^2)\left(|\dbar^\star\xi|^2\right)e^{-\varphi_L}dV_{\omega_{\cC}}
\end{equation}
by Proposition \ref{unifor1}.
\smallskip

\noindent The rest of the integral \eqref{fin9} is analysed as follows. For simplicity we assume that 
$V_{\rm sing}= \Omega$ is a coordinate subset and the expression we have to evaluate is
\begin{equation}\label{eq121}
\left| \frac{1}{\ep^2}\int_{\Omega}\rho'\left(\frac{|s_Y|^2}{\ep^2}\right)U_0\wedge \ol{\partial s_Y\wedge \gamma_{\xi_1} }e^{-\varphi_L}dV_{{\omega_{\cC}}}\right|.
\end{equation}
This is bounded by the quantity
\begin{equation}\label{fin1111} 
\sup_\Omega(|\xi_1|^\alpha _{\varphi_L}) \left| \frac{1}{\ep^2}\int_{\Omega}\rho'\left(\frac{|s_Y|^2}{\ep^2}\right)   
|\partial s_Y|^\alpha |U_0|_{\omega_{\cC}} |\partial s_Y\wedge \gamma_{\xi_1}|^{1-\alpha}e^{-(1-\alpha/2)\varphi_L}dV_{{\omega_{\cC}}}\right|
\end{equation}
and H{\"o}lder inequality shows that \eqref{fin1111} is smaller than the product of 
\begin{equation}\label{fin12} 
\sup_\Omega(|\xi_1|^\alpha  _{\varphi_L}) \left(\int_{\Omega\cap K_\ep} 
|\partial s_Y\wedge \gamma_{\xi_1}|^{2}\frac{e^{-\varphi_L}}{\ep^2}dV_{{\omega_{\cC}}}\right)^{\frac{1-\alpha}{2}}
\end{equation}
with 
\begin{equation}\label{fin13}
\left(\int_{\Omega} \rho'\left(\frac{|s_Y|^2}{\ep^2}\right)  
|\partial s_Y|^{\frac{2\alpha}{1+\alpha}} |U_0|_{\omega_{\cC}}^{\frac{2}{1+\alpha}}
\frac{e^{-\frac{\varphi_L}{1+\alpha}}}{\ep^2}dV_{{\omega_{\cC}}}\right)^{\frac{1+\alpha}{2}}
\end{equation}

\noindent The limit of the quantity \eqref{fin13} as $\ep \to 0$ is equal to 
\begin{equation}\label{fin14}
\left(\int_{\Omega\cap Y}   
\left|\frac {u}{\partial s_Y}\right|_{\omega_{\cC}}^{\frac{2}{1+\alpha}}e^{-\frac{\varphi_L}{1+\alpha}}dV_{{\omega_{\cC}}}\right)^{\frac{1+\alpha}{2}}.
\end{equation}
As for the \eqref{fin12}, we use Theorem \ref{meanV} together with our previous considerations and it follows that it is smaller than
\begin{equation}\label{fin15}
C\left(\int_\Omega
|\xi_1|^2 e^{-\varphi_L}dV_{\omega_{\cC}}\right)^{\alpha/2} \left(\int_X\prod\log^2(\ep^{2/N}+ |s_{j}|^2)\left(|\dbar^\star\xi|^2\right)e^{-\varphi_L}dV_{{\omega_{\cC}}}\right) ^{\frac{1-\alpha}{2}}.
\end{equation}

\noindent It is at this point that we are using the positivity assumption (i): we have
\begin{equation}\label{eq123}
  \int_{V_{\rm sing}}|\xi_1|^2e^{-\varphi_L}dV_{\omega_{\cC}}\leq \frac{1}{C_{\rm sing}}\int_{V_{\rm sing}}
  \big\langle [\Theta_{h_L}(L), \Lambda_{\omega_{\cC}}]\xi_1, \xi_1\big\rangle e^{-\varphi_L}dV_{\omega_{\cC}},
\end{equation}
where $C_{\rm sing}$ is the (positive) lower bound for the positivity of $\displaystyle (L, h_L)|_{V_{\rm sing}}$.
By Bochner formula we get
\begin{equation}\label{eq124}
\int_{V_{\rm sing}}|\xi_1|^2e^{-\varphi_L}dV_{{\omega_{\cC}}}\leq \frac{1}{C_{\rm sing}} \int_X\left|\dbar^\star\xi_1\right|^2e^{-\varphi_L}dV_{{\omega_{\cC}}} .
\end{equation}
\medskip

\noindent Thus we obtain the expected estimate for the functional \eqref{fin9}
and Theorem \ref{otconique} is proved. We remark that the contribution of the singularities of $Y$ to the estimate in this result is
\begin{equation}\label{est0}
C\left(1+ \frac{1}{C_{\rm sing} ^\alpha}\right) \left(\int_{\Omega\cap Y}   
\left|\frac {u}{\partial s_Y}\right|_{\omega_{\cC}}^{\frac{2}{1+\alpha}}e^{-\frac{\varphi_L}{1+\alpha}}dV_{{\omega_{\cC}}}\right)^{1+\alpha}
\end{equation}
where $C$ is a constant depending on $(X, V_{\rm sing}, {\omega_{\cC}})$ and $\alpha\in [0, 1]$ is an arbitrary positive real which is smaller than 1.   
\medskip

\begin{remark} The quantity \eqref{est0} is part of the term estimating 
 \begin{equation}\label{est1}
 \int_{X\setminus V_{\rm sing}}\frac{|U|^2}{|s_Y|^2\prod\log^2(|s_{j}|^2)}e^{-\varphi_L-\varphi_Y}dV_{\omega_{\cC}}.
 \end{equation}
A slight modification of the proof shows that we can get a similar estimate for the integral
\begin{equation}
 \int_{X\setminus V_{\rm sing}}\frac{|U|^2}{|s_Y|^2\log^{2+\tau}(1/|s_{Y}|^2)}e^{-\varphi_L-\varphi_Y}dV_{\omega_{\cC}}
 \end{equation}
 for any strictly positive real $\tau$. 
\end{remark}

\medskip

\begin{remark}
Actually one can replace the curvature condition (i) with the following:
\emph{there exists a constant $C_{\rm sing}> 0$ such that we have}
\begin{equation} 
\Theta_{h_L}(L)\geq \frac{C_{\rm sing}}{\log\frac{1}{|s_Y|^2}}{\omega_{\cC}}
\end{equation}
\emph{pointwise on $V_{\rm sing}$}. The estimate for the extension
we obtain in the end is the same, but we are using a twisted Bochner formula instead of \eqref{eq124}. 
\end{remark}


\subsection{Proof of Theorem \ref{ot}} In this subsection $\omega$ is a fixed reference
K\"ahler metric on $X$ (in particular, non-singular).

\noindent By hypothesis, the metric $h_L$ is
non-singular and in this case the equality
\begin{equation}\label{fin20}
\int_Y\frac{u}{ds_Y}\wedge \ol{\gamma_{\xi}}e^{-\varphi_L}= \int_Y\frac{u}{ds_Y}\wedge \ol{\gamma_{\xi_1}}e^{-\varphi_L}
\end{equation}   
is immediate.

We remark that the restriction $\displaystyle \frac{u}{ds_Y}\Big|_{Y_j}$ is holomorphic, for each component $Y_j$ of $Y$. This is where the vanishing of $u$ on the singularities of $Y$ is used. We decompose the restriction of $\displaystyle \gamma_{\xi_1}$ to $Y_j$ as follows
\begin{equation}\label{fin21}
\gamma_{\xi_1}|_{Y_j}= \alpha_j+ \beta_j
\end{equation}
where $\alpha_j$ is holomorphic and $\beta_j$ is orthogonal to the space of $L$-valued holomorphic top forms on $Y_j$. Then we have
\begin{equation}\label{fin22}
\int_{Y_j}\frac{u}{ds_Y}\wedge \ol{\gamma_{\xi_1}}e^{-\varphi_L}=
\int_{Y_j}\frac{u}{ds_Y}\wedge \ol{\alpha_j}e^{-\varphi_L}
\end{equation}  
Let $x_0$ be a singular point of $Y$. We then have coordinates $(z_1, z_2)$ defined on a open subset $x_0\in V$, centered at $x_0$ and such that $(z_k= 0)= Y_k\cap V$ for each $k=1, 2$. We equally fix a trivialization of $L|_V$ and let 
$\varphi_L$ be the corresponding weight of the metric $h_L$. We write
\begin{equation}\label{fin23}
u|_{V}= f_udz_1\wedge dz_2\otimes e_L
\end{equation}
and let $\theta$ be a function which is equal to 1 near $x_0$ and such that $\Supp(\theta)\subset V$.

\noindent Thus we have 
\begin{equation}\label{fin24}
\frac{u}{ds_Y}\big|_{V\cap Y_1}= f_u\frac{dz_{2}}{z_{2}}\otimes e_L
\end{equation}
together with a similar equality on $V\cap Y_2$. We can write
\begin{align}\label{fin25}
\partial_{\varphi_L}\left(\theta f_u\log|z_2|^2\otimes e_L\right)= & \theta f_u\frac{dz_{2}}{z_{2}}\otimes e_L \\
+ & \theta \log|z_2|^2\partial_{\varphi_L}(f_u\otimes e_L)+ 
f_u \log|z_2|^2\partial\theta\otimes e_L\nonumber \\
\nonumber
\end{align}
and then we observe that the left hand side term of \eqref{fin25} is $\partial_{\varphi_L}$--exact
on $Y_1$. Therefore we have 
\begin{equation}\label{fin26}
\int_{Y_1}\partial_{\varphi_L}\left(\theta f_u\log|z_2|^2\otimes e_L\right)\wedge \ol{\alpha_1}e^{-\varphi_L}= 0
\end{equation}
since $\alpha_1$ is holomorphic. We infer that we have 
\begin{align}\label{fin27}
-\int_{Y_1}\theta\frac{u}{ds_Y}\wedge \ol{\alpha_1}e^{-\varphi_L}= & 
\int_{Y_1}\theta \log|z_2|^2\partial_{\varphi_L}(f_u\otimes e_L)\wedge \ol{\alpha_1}e^{-\varphi_L}\nonumber\\
+ &\int_{Y_1}f_u \log|z_2|^2\partial\theta\otimes e_L\wedge \ol{\alpha_1}e^{-\varphi_L}
\nonumber\\
\nonumber
\end{align}
and all that we still have to do is to apply the Cauchy-Schwarz inequality to each of the two terms of the RHS of the inequality above.

\noindent A last remark is that we have 
\begin{equation}\label{fin28}
\int_{Y_1}|\alpha_1|^2e^{-\varphi_L}\leq \int_{Y_1}|\gamma_{\xi_1}|^2e^{-\varphi_L}
\end{equation}
by the definition of $\alpha_1$ and $\beta_1$. We use the a-priori inequality and we conclude as in Theorem \ref{otconique}.\qed
\medskip

\begin{remark} In the absence of hypothesis
  $u|_{Y_{\rm sing}}\not\equiv 0$ the evaluation of the term \eqref{fin20} near the singularities of $Y$ is problematic. In the decomposition \eqref{fin21}, we write $\beta_j= \dbar^\star(\tau_j)$,
  and then the question is to estimate the quotient $$f_j:= \frac{\tau_j}{\omega} $$ at the points of $Y_{\rm sing}$. This does not seem to be possible, since we only have the norm $W^{1,2}$ of $f_j$ at our disposal. Indeed, the quantity
  $\dbar \beta_j$ is equal to the restriction of the form
  $\displaystyle \dbar \gamma_{\xi_1}$ to $Y_j$. In Question \ref{quest1001} we provide a few more 
 precisions about this matter. 
  
\end{remark}
\medskip


\subsection{Proof of Theorem \ref{otflat}}
  This is another set-up in which the considerations above work, as follows.
  We recall that the metric of $h_L$ of $L$ satisfies the hypothesis (a) and (b) at the beginning \emph{and} moreover $(L, h_L)$ is flat near the singularities of $Y$, i.e.
\begin{equation}\label{flat1}
\Theta_{h_L}(L)\big|_{V_{\rm sing}}= 0.
\end{equation}
Then we get an estimable extension as follows. Let $\omega$ be a fixed
K\"ahler metric on $X$. As in the proof of the preceding result Theorem \ref{ot},
we will use the method of Berndtsson \cite{BB1}, so the quantity to be bounded is
\begin{equation}\label{flat2}
\int_{Y\cap V_{\rm sing}}\frac{u}{ds_Y}\wedge \ol{\gamma_{\xi_1}}e^{-\varphi_L}.
\end{equation}  
Integration by parts shows that it is enough to obtain a mean value
inequality for the function
\begin{equation}\label{flat3}
\sup_{1/2V_{\rm sing}}|\dbar \gamma_{\xi_1}|^2= \sup_{1/2V_{\rm sing}}|\partial^\star_{\varphi_L}\xi_1|^2.
\end{equation}
This is done according to the same principle as before. In the first place the differential inequality satisfied by $\displaystyle |\partial^\star_{\varphi_L}\xi_1|^2$ is as follows
\begin{equation}\label{flat4}
\Delta''\left(|\partial^\star_{\varphi_L}\xi_1|^2\right)\geq |\nabla \left(\partial^\star_{\varphi_L}\xi_1\right)|^2- C|\partial^\star_{\varphi_L}\xi_1|^2
\end{equation}
for some constant $C> 0$ which only depends on (the curvature of) $\displaystyle\omega |_{V_{\rm sing}}$. We will not detail the calculation here because this is very similar with the one in the proof of Theorem \ref{otconique}. However, we highlight next the main differences:
\begin{enumerate}
\smallskip

\item[(1)] It is not necessary to introduce any regularization of the metric,
  since by hypothesis \eqref{flat1} the restriction $\displaystyle h_L|_{V_{\rm sing}}$ is non-singular.

\smallskip

\item[(2)] Without any additional information about $(L, h_L)$, the term $\displaystyle \left\langle \partial^\star_{\varphi_L}\xi_1, \square \partial^\star_{\varphi_L}\xi_1\right\rangle $ is problematic. Actually \eqref{flat1} is needed precisely in order to deal with it: it the curvature of $(L, h_L)|_{V_{\rm sing}}$ equals zero, then we have $\square \partial^\star_{\varphi_L}\xi_1= 0$ pointwise on $V_{\rm sing}$. In general we have the term
\begin{equation}\nonumber
\left\langle [\dbar, \Lambda_{\Theta_{h_L}(L)}]\xi_1, \partial^\star_{\varphi_L}\xi_1\right\rangle
\end{equation}
which appears in the computation and seems impossible to manage. 
\smallskip

\item[(3)] In the evaluation of the Laplacian of the norm of a harmonic tensor we have two terms: the gradient of the tensor, and several curvature terms
  corresponding to the metric on the ambient manifold and to the twisting, respectively. Here we don't have any contribution from $L$, and the term involving the curvature of $\omega$ is taken care by the constant $-C$ in
  \eqref{flat4}.
\end{enumerate}

\smallskip

\noindent Anyway, the inequality \eqref{flat4} can be re-written as
\begin{equation}\label{flat8}
  \Delta''\left(|\partial^\star_{\varphi_L}\xi_1|^2\right)
  \geq \big|\nabla \left|\partial^\star_{\varphi_L}\xi_1\right|\big|^2- C|\partial^\star_{\varphi_L}\xi_1|^2
\end{equation}
and this
combined with Moser iteration procedure shows that we have
\begin{equation}\label{flat6}
\sup_{1/2V_{\rm sing}}|\partial^\star_{\varphi_L}\xi_1|_{\omega, h_L}^{\alpha}\leq C\int_{V_{\rm sing}}|\partial^\star_{\varphi_L}\xi_1|^{\alpha}_{\omega, h_L}dV_\omega.
\end{equation}

Finally, the term that one (almost) never uses in Bochner formula shows that we have
\begin{equation}\label{flat7}
  \int_X|\partial^\star_{\varphi_L}\xi_1|^2_{\omega, h_L}dV_\omega\leq
  \int_X|\dbar^\star\xi_1|^2_{\omega, h_L}dV_{\omega}
\end{equation}
and we thus obtain the inequality
\begin{equation}
\sup_{1/2V_{\rm sing}}|\partial^\star_{\varphi_L}\xi_1|_{\omega, h_L}^{\alpha}\leq C\int_{V_{\rm sing}}|\dbar^\star_{\varphi_L}\xi_1|^{\alpha}_{\omega, h_L}dV_\omega.  
\end{equation}
Then we conclude as in Theorem \ref{ot}.\qed
\medskip

\begin{remark}
Actually in the proof of Theorem \ref{otflat} only needs to evaluate the $L^2$ norm 
\begin{equation}
\int_{Y}|\dbar \gamma_{\xi_1}|^2dV_{\omega}
\end{equation}
of $\dbar \gamma_{\xi_1}|_{Y}$. One might try to use a similar method as the one in section 2, but there are serious difficulties to overcome.    
\end{remark}
\medskip


\subsection{Proof of Theorem \ref{otflat1}} 
By hypothesis we know that $Y$ has one component $Y_1$ 
which only intersects $\displaystyle \cup_{i\neq 1} Y_i$ in a unique point
$p_0$ such that $u(p_0)\neq 0$. We also assume that $\displaystyle L|_{Y_1}$ is flat, in the sense that
  there exists a section $\tau$ 
  such that $\tau(p_0)\neq 0$ and $\displaystyle \partial_{\varphi_L} \tau=0$.
  Then we argue as follows.
  \smallskip
  
\noindent Let $\omega$ be a fixed, reference metric on $X$. On each component $Y_j$ of $Y$ we solve the equation 
\begin{equation}\label{po10}
\gamma_{\xi_1}|_{Y_j}= \alpha_j+ \dbar^\star\beta_j
\end{equation}
where $\alpha_j$ is holomorphic $(1,0)$ form and $\beta_j$ is of type $(1,1)$ on $Y_j$. We note that
by elliptic regularity the form $\beta_j$ is smooth. 

\noindent We have $\displaystyle \beta_j= f_j\omega|_{Y_j}$ and then the equality
\begin{equation}\label{po11}
\int_{Y_j}\big\langle\frac{u}{\sigma_jds_{Y_j}}, \dbar^\star\beta_j\big\rangle_{\omega} e^{-\varphi_L}dV_{\omega}= \sum_{x\in Y_{\rm sing}\cap Y_j} f_u(x)\ol{f_j(x)}e^{-\varphi_L(x)}
\end{equation}      
follows by the residues formula. Here we denote by $\displaystyle \sigma_j:= \prod_{i\neq j} s_{Y_i}$.

\noindent In case $j=1$, the sum above only has one term, by hypothesis. Since we have 
$\dbar^\star (\tau\omega)= 0$, we can modify the solution $f_1$ so that the global sum of residues is zero.

\begin{question}\label{quest1001}{\rm Let $p$ be one of the intersection points of two curves
$Y_1\cap Y_2$ in $X$. The analogue of the a-priori inequality in section 2 gives
\begin{equation}\label{qe2}
  |f_j(p)|^2e^{-\varphi_L(p)}\leq C
  \int_X\frac{\log^2\Vert s\Vert^2}{\Vert s\Vert^{2}}|\partial_{\varphi_L} f_j|^2e^{-\varphi_L}dV_\omega
\end{equation}
provided that the bundle $(L, h_L)$ has the right curvature hypothesis, let us assume this holds for the moment. In \eqref{qe2} we denote by  $C$ a constant which we can compute 
explicitly.
This a-priori inequality is obtained by considering the $(2,2)$--form with values in $L$
\begin{equation}\label{qe3}
 f_j\omega^2
 \end{equation}
 whose $\star$ coincides with the section $f_j$, and use the procedure Theorem 2.1
 for the function $\displaystyle w:= \frac{1}{|s_1|^2+ |s_2|^2}$. The curvature requirements this induces will most likely be 
\begin{equation}\label{qe4}
\frac{1}{\delta}\Theta_{h_L}(L)\wedge \omega\geq \frac{|s_1|^2\Theta(Y_1)+ |s_2|^2\Theta(Y_2)}{|s_1|^2+ |s_2|^2}\wedge \omega.
\end{equation} 
 We have 
\begin{equation}\label{qe5}
 \dbar f_j\omega^2= 0, \qquad \dbar^\star  f_j\omega^2= -\star \left(\partial_{\varphi_L}f_j \right)
 \end{equation}
which explains \eqref{qe2}. 

\smallskip
 
\noindent Anyway, by equality \eqref{po10} we control the norm
\begin{equation}
\int_Y|\partial_{\varphi_L} f_j|^2e^{-\varphi_L}dV_\omega.
\end{equation}
Then the question is:
}can we find a smooth section $\wt f_j$ of $L$ such that it equals $f_j$ on $Y$
and such that
\begin{equation}\label{qe1}
  \int_X\frac{\log^2\Vert s\Vert^2}{\Vert s\Vert^{2}}|\partial_{\varphi_L} \wt f_j|^2e^{-\varphi_L}dV_\omega\leq C\int_Y|\partial_{\varphi_L} f_j|^2e^{-\varphi_L}dV_\omega,
\end{equation}
where $C$ in \eqref{qe1} is universal?
\end{question}


\medskip

\appendix
\section{Further results and examples\\ \rm{(by Bo Berndtsson)}}\label{BoB}



\noindent In this appendix we will study two very simple model examples of $L^2$-extension from a non reduced or singular variety. It is the second example (in section 3) that is most relevant to the subject of the main paper. The main point is to show that it is not possible to obtain an estimate that is substantially better than Theorem \ref{otconique} and Theorem \ref{ot}. The role of the first example is to show that similar difficulties appear already in an even simpler situation, that can be analyzed in a more complete way.

In the first example we consider the space
$$
A^2_\phi(\Delta)=\{h\in H(\Delta); \int_\Delta |h|^2 e^{-\phi} d\lambda:=\|h\|^2<\infty \}
$$
of holomorphic functions in the disk $\Delta$  that are square integrable against a  weight $e^{-\phi}$. Given numbers $a_k$, $k=0, 1, ...N-1$, we will compute the minimal norm of a function $h\in A^2$ that satisfies $h^{(k)}(0)=a_k$. The formula we give is exact but not very explicit; it contains the Bergman kernel for the space and various metrics derived from it.  Therefore we will also discuss to what extent it is possible to estimate it in more concrete terms, and show that the most optimistic estimates fail.

In the second example we consider the unit polydisk $U$ in $\C^2$ and the singular variety $V=\{z\in U; z_1z_2=0\}$ in $U$. We again denote by $A^2_\phi(U)$ the Bergman space of  holomorphic functions in $U$ that are square integrable against the  weight $e^{-\phi}$. The extension problem is now to find a function $h\in A^2$ that restricts to a given function on $V$, i. e. satisfies $h=f_1$ when $z_1=0$ and $h=f_2$ when $z_2=0$, where $f_1$ and $f_2$ are holomorphic functions of one variable satisfying $f_1(0)=f_2(0)$. In this case we get only an {\it estimate} for the minimal extension, and we conclude by an example (basically due to Ohsawa, \cite{Oh7}) which (perhaps) can serve as a motivation for the statement in Theorem 1.4.  

\subsection{Extension from a (fat) point in the unit disk}

Let $A^2=A^2_\phi(\Delta)$ be defined as above and let for $k=0, 1, 2...$
$$
E_k=\{h\in A^2; h^{(j)}(0)=0, j<k \}.
$$
It follows from elementary Hilbert space theory that there is a unique function $h$ of minimal norm in $A^2$ satisfying $h^{(k)}(0)=a_k$ for $k=0, ...N-1$. Write
$$
h=h_0+r_1
$$
where $r_1\in E_1$ and $h_0\perp E_1$. Then we write
$$
r_1=h_1+r_2,
$$
with $r_2\in E_2$ and $h_1\perp E_2$. Continuing this way we get
$$
h=h_0+h_1+... h_{N-1} +r_N,
$$
with $h_k\in E_k\ominus E_{k+1}$ and $r_N$ in $E_N$. That $h$ has minimal norm means that $h$ is orthogonal to $E_N$, so $r_N=0$. By orthogonality we have
$$
\|h\|^2=\sum _0^{N-1} \|h_k\|^2,
$$
so the problem amounts to estimating the norms of $h_k$.

The spaces $E_k\ominus E_{k+1}$ are one dimensional. Let $ e_k$ be an element of unit length.
We start with a simple lemma from \cite{HW} whose proof follows almost directly from the definitions.
\begin{lemma}
  $$
  |e_k^{(k)}(0)|^2 =\sup_{f\in E_k} \frac{|f^{(k)}(0)|^2}{\|f\|^2}.
    $$
\end{lemma}
For $k=0$, $|e_0(0)|^2=B_0(0)$, the (diagonal) Bergman kernel at the origin. For $k\geq 1$ $|e_k^{(k)}(0)|^2=:B_k(0)$  can be viewed as a 'higher order Bergman kernel' and we refer to \cite {HW} for interesting applications of this idea. By a classical formula of Bergman, \cite{Berg}, we have
\begin{equation}\label{2.1}
B_1(0)=|e_1'(0)|^2 = B_0(0)\frac{\partial^2\log B_0(z)}{\partial z\partial \bar z}=:B_0\omega_B,
\end{equation}
so the first order Bergman kernel is strongly related to the Bergman {\it metric}. (We are abusing notation by identifying the metric with its density; more properly we should write
$$
\omega_B=\frac{\partial^2\log B_0(z)}{\partial z\partial \bar z}idz\wedge d\bar z.)
$$

Since $h^{(k)}(0)=a_k$ and $h_j^{(k)}(0)=0$ for $j>k$, we get
$$
h_0(0)=b_0:=a_0, \quad h_1'(0)=b_1:= a_1-h_0'(0), \quad h_2''(0)=b_2:= a_2-h_1''(0)-h_0''(0), ...
$$
Then, since $h_k$ is a multiple of $e_k$ and $h_k^{(k)}(0)=b_k$, we have 
$$
h_k=\frac{b_k}{e_k^{(k)}(0)}e_k.
$$
Recalling that $e_k$ has norm 1 and that $|e_k^{(k)}(0)|^2=B_k(0)$ we find that the norm of the minimal extension is given by
\begin{equation}\label{2.2}
\|h\|^2=\sum_0^{N-1} |b_k|^2/B_k(0).
\end{equation}
When $N=1$ this is just the standard formula
$$
\|h\|^2=|a_0|^2/B_0,
$$
which shows that estimates from above of the norm of the minimal extension are equivalent to estimates from below of the (usual) Bergman kernel. The next case is $N=2$. Then we use \eqref{2.1} and find (since $h_0=(a_0/e_0(0)) e_0$) 
$$
\|h\|^2= (|a_0|^2 +|a_1-a_0 e_0'(0)/e_0(0)|^2_{\omega_B})/B_0(0).
$$
Here we think of $a_1-a_0 e_0'(0)/e_0(0)$ as a 1-form and the second term in the right hand side  is its norm for the Bergman metric. Since the off-diagonal Bergman kernel $B_0(z,w)$ is holomorphic in $z$ and antiholomorphic in $w$, and $e_0(z)=B_0(z,0)$,  we have
$$
e_0'(0)/e_0(0)=(\partial/\partial z)|_0\log B_0(z,z).
$$
Hence we can also write
$$
\|h\|^2= (|a_0|^2 +|a_1-a_0 \partial\log B_0(0)|^2_{\omega_B})/B_0(0).
$$
By the standard Ohsawa-Takegoshi theorem, $B_0(0)\geq C^{-1} e^{\phi(0)}$, with a universal constant $C$, so we get the slightly more explicit estimate
\begin{equation}\label{2.3}
\|h\|^2\leq C (|a_0|^2 +|a_1-a_0 \partial\log B_0(0)|^2_{\omega_B})e^{-\phi(0)}.
\end{equation}
Because of the following lemma we can replace the norm with respect to the Bergman metric by the Euclidean norm.
\begin{lemma} For any (1-form) $a$, we have at the origin of $\Delta$ 
  $$
  | a|^2_{\omega_B}\leq |a|^2.
  $$
\end{lemma}
\begin{proof} 
  By \eqref{2.1} and Lemma A.1, at the origin
  $$
  \omega_B=B_1(0)/B_0(0)= \sup_{f\in E_1}\frac{|f'(0)|^2}{\|f\|^2 B_0(0)}.
    $$
    Now choose $f=ze_0$. Since $|z|<1$ and $e_0$ has norm 1, $f$ has norm less that 1. Moreover, $f'(0)= e_0(0)$. Hence $\omega_B\geq 1$, which proves the lemma.
    \end{proof}

Thinking of $\log B_0(z)$ as an approximation of $\phi$,  \eqref{2.3} suggests that  one might also have the inequality
$$
\|h\|^2\leq C (|a_0|^2 +|a_1-a_0 \partial\phi(0)|^2e^{-\phi(0)},
$$
but we shall see later that this does not hold. 

We next discuss briefly   estimates for larger values of $N$. We first note that there is a version of Lemma A.2 for all $k$, which is proved in much the same way.
\begin{lemma}
  $$
  B_k(0)\geq (k!)^2 B_0(0).
  $$
\end{lemma}
\begin{proof} Recall that
    $$
  B_k(0)= |e_k(0)|^2=\sup_{f\in E_k} \frac{|f^{(k)}(0)|^2}{\|f\|^2}.
    $$
  Take $f=z^k e_0$. Then $f$ has norm less than 1 and $f^{(k)}(0)= k! e_0(0)$. Thus
  $$
  B_k(0)\geq (k!)^2 |e_0(0)|^2=(k!)^2 B_0(0).
  $$
\end{proof}
From the lemma and \eqref{2.2} we get the estimate for the norm of the $L^2$-minimal extension
$$
\|h\|^2\leq (\sum_{k=0}^{N-1} |b_k|^2)/B_0(0)\leq C(\sum |b_k|^2)e^{-\phi(0)},
$$
where
$$
b_k=a_k-\sum_{j=0}^{k-1} h_j^{(k)}(0)=a_k-\sum_{j=o}^{k-1}\frac{b_j}{e_j(0)} e_j^{(k)}(0).
$$
As we have seen, this is difficult to estimate even when $N=2$ and it is clear that the complexity grows with higher values of $k$ and $N$.

\subsection{Examples}
We focus on the estimates for $N=2$, i. e. extension of a first order jet. The most naive conjecture for an explicit estimate would be 
\begin{equation}\label{2.4}
\|h\|^2\leq C(|a_0|^2 +|a_1|^2) e^{-\phi(0)}.
\end{equation}
\medskip

{\bf Claim 1:}{\it There is no constant $C$ independent of $\phi$ such that for all subharmonic $\phi$, \eqref{2.4} holds.}

For this, take $\phi(z)=-2m\Re(z)$ and put $g=e^{mz}h$. Take $a_0=1, a_1=0$. If \eqref{2.4} held we would get
$$
\|h\|^2\leq C e^{-\phi(0)}=C.
$$
Hence
$$
\int_\Delta |g|^2 d\lambda=\|h\|^2\leq C,
$$
and $ g'(0)=m$. This contradicts Cauchy's estimates for the derivative.

\medskip

The estimate
\begin{equation}\label{2.5}
\|h\|^2\leq C(|a_0|^2 +|a_1-a_0\partial\phi(0)|^2) e^{-\phi(0)}.
\end{equation}
might seem  more plausible since we estimate $\|h\|^2$ by the connection in $A^2_\phi$, $h'-h\partial\phi$  instead of just $h'(0)$. 

{\bf Claim 2: }{\it There is no constant $C$ independent of $\phi$ such that for all subharmonic $\phi$ the minimal extension satisfies \eqref{2.5}.}

Here we take $\psi=\max(\phi +\epsilon\log |z|^2, -A)$, where $A$ is a large constant. If \eqref{2.5} holds for $\psi$, then \eqref{2.4} also holds for $\psi$ since $\partial\psi$ vanishes near the origin. Letting $\epsilon \to 0$ we get an extension that satisfies (2.4) for $\phi$ (assuming $\phi\geq -A$ in the disk). By the first claim, this is impossible.

\subsection{A singular variety in the bidisk}
In this section we study $L^2$-extension from the variety
$$
V=\{z\in U; z_1z_2=0 \}
$$
in the unit bidisk $U$, and we use the notation from the introduction. Following the scheme in the previous section we let
$$
E_1=\{h\in A^2_\phi(U); h(0)=0\},
$$
and
$$
E_2=\{h\in A^2_\phi(U); h|_V=0\}.
$$
Let $f$ be a holomorphic function on $V$, and let $h$ be the holomorphic extension of $f$ to $U$ of minimal norm. Again we write
$$
h=h_0 + r_1,
$$
where $h_0\perp E_1$ and $r_1\in E_1$. Continuing as before we write
$$
r_1=h_1+ r_2
$$
where $h_1\in E_1\ominus E_2$ and $r_2\in E_2$. Then $h=h_0+h_1+ r_2=h_0+h_1$ since, by minimality,  $h$ is orthogonal to $E_2$. Moreover, $h_0\perp h_1$, so
$$
\|h\|^2=\|h_0\|^2 +\|h_1\|^2.
$$
The holomorphic function $f$ on $V$ is given by a pair $(f_1,f_2)$ where $f_1$ is holomorphic on $\{z_1=0\}$ and $f_2$ is holomorphic on $\{z_2=0\}$, and $f_1(0)=f_2(0)=: a_0$.

Since $h_0$ is orthogonal to $E_1$ and $h_0(0)=a_0$ we have that
$$
h_0= \frac{a_0}{e_0(0)} e_0,
$$
where as in the previous section $e_0$ is a function of unit norm orthogonal to $E_1$. Then $|e_0(0)|^2=B_0(0)$, the (diagonal) Berman kernel at the origin, and we get
$$
\|h_0\|^2=|a_0|^2/B_0(0),
$$
just as before.

We next turn to $h_1$, which is the $L^2$-minimal extension of $\tilde f:=f-h_0$. We can not give an exact formula for the norm of $h_1$ but it is easy to give an estimate. Since $\tilde f$ vanishes at the origin we have
$\tilde f=(f_1-h_0, f_2-h_0)=:(z_2 g_1, z_1 g_2)$, where $g_1(z_2)$ and $g_2(z_1)$ are holomorphic functions of one variable. Let $G_1$ and $G_2$ be the minimal extensions of $g_1$ and $g_2$, from $V_1=\{z_1=0\}$ and $V_2=\{z_2=0\}$ respectively. By the Ohsawa-Takegoshi theorem
$$
\|G_i\|^2\leq C\int_{V_i} |g_i|^2 e^{-\phi} d\lambda.
$$
Let $H=z_2 G_1+z_1 G_2$. This is an extension of $\tilde f$ and
$$
\|H/|z|\|^2 \leq \|G_1\|^2 +\|G_2\|^2\leq C \int_V |f-h_0|^2/|z|^2 e^{-\phi} d\lambda.
$$
Hence
$$
\|H\|^2\leq \|H/|z|\|^2\leq C\int_V |f-a_0|^2/|z|^2 e^{-\phi} d\lambda.
$$
All in all we get the estimate for the minimal extension  $h$ of $f$, 
\begin{equation}
\|h\|^2\leq C( |a_0|^2/B_0(0) +\int_V |f-h_0|^2/|z|^2 e^{-\phi} d\lambda).
\end{equation}
Here one might hope that the  quotient $ \frac{|f-h_0|^2}{|z|^2}$ in the right hand side could be replaced by  the squared norm of a derivative acting on $f$. Asymptotically as $z\to 0$ on e. g. $V_2$,
$$
\frac{f-h_0}{z_1} \to f'(0)-h_0'(0)=f'(0)-a_0 \frac{e_0'(0)}{e_0(0)} =\frac{\partial f(0)}{\partial z_1} -a_0\frac{\partial\log B_0(0)}{\partial z_1}
$$
(the last equality follows as in the discussion leading to \eqref{2.3}). Again, thinking of the logarithm of the Bergman kernel, $\log B_0$ as  an approximation of $\phi$, 
one is led to look for estimates in terms of $\partial^\phi f$, like in Theorem \ref{ot}. 
Theorem \ref{ot} however also contains a factor $\log^2(\max  |z_j|^2)$, and we next give an example showing that something of this kind is necessary.

\subsection{More examples}
We first give a counterexample (cf. \cite{Oh7}) to the most naive conjecture; that the same estimate as for smooth varieties holds.

\noindent {\bf Claim 3:} {\it  There is no universal constant, independent of the plurisubharmonic weight $\phi$, such that the minimal extension satisfies
\begin{equation}
\|h\|^2\leq C\int_V |f|^2 e^{-\phi} d\lambda
\end{equation}
for all functions $f$ that vanish at the origin.}

To see this, take $f=z_1$ when $z_2=0$ and $f=0$ when $z_1=0$. Take $\phi=\log |z_1-z_2|^2$. Any extension $H$ must have the form $H=z_1 G$. If $H$ has finite norm, then $G=0$ when $z_1=z_2$. Hence $H$ vanishes to second degree at the origin, which is not possible.

\medskip

Actually, this same example shows that it does not help to add the $L^2$-norm of the twisted derivative of $f$ in the right hand side. We use the notation
$$
\partial^\phi f=e^\phi\partial e^{-\phi} f=\partial f-f\partial\phi.
$$ 
{\bf Claim 4:} {\it There is no universal  constant, independent of the plurisubharmonic weight $\phi$, such that the minimal extension satisfies
\begin{equation}\label{3.3}
\|h\|^2\leq C(\int_V |f|^2 e^{-\phi} d\lambda +\int_V |\partial^\phi f|^2 e^{-\phi} d\lambda)
\end{equation}
for all functions $f$ that vanish at the origin.} 

Indeed, with the same choice of $\phi$ and $f$ as above, we have on $V_2$ outside the origin
$$
\partial^\phi f= dz_1-z_1(1/z_1) dz_1=0.
$$
Since the weight $\phi$ has a singularity at the origin we look at regular approximations. Let $\phi_\epsilon$ be the convolution of $\phi$ with $(\pi\epsilon^2)^{-1}\chi_\epsilon$, where $\chi_\epsilon$  is the characteristic function of the disk with radius $\epsilon$. Explicitly,
$$
\phi_\epsilon(\zeta)=\frac{|\zeta|^2-\epsilon^2}{\epsilon^2} +\log\epsilon^2
$$
when $|\zeta|<\epsilon$ and $\phi_\epsilon(\zeta)=\log|\zeta|^2$ when $|\zeta|\geq \epsilon$. This gives a sequence of subharmonic functions decreasing to $\phi$ on $V_2$.  We have
$$
\partial\phi_\epsilon= \frac{\bar\zeta}{\epsilon^2}d\zeta
$$
when $|\zeta|<\epsilon$. Hence
$$
\partial^{\phi_\epsilon}f=\frac{\epsilon^2-|z_1|^2}{\epsilon^2}\chi_\epsilon dz_1
$$
on $V_2$. Since $e^{-\phi_\epsilon}$ is of size roughly $\epsilon^{-2}$ when $|\zeta|<\epsilon$, the right hand side in \eqref{3.3} stays bounded as $\epsilon\to 0$. Hence, if \eqref{3.3} held, we would  again get an extension of finite norm in $L^2(e^{-\phi})$, which we have seen is impossible. This motivates the logarithmic factor in the estimate \eqref{log} of Theorem \ref{ot}. 
\medskip

\noindent It is easy to construct a compact analogue of the example used to
prove the claims 3 and 4 above, as we briefly indicate next.

Consider two transverse lines $L_1$ and $L_2$ in $X:= \bP^2$, and the adjoint bundle $K_X+ L_1+ L_2+ \O(2)$. We define the metric $h$ on $\O(2)$ by its weights
\begin{equation}\label{eq40}
\varphi_\ep:= \log(\ep^2e^{\phi_{\rm FS}}+ |f_1-f_2|^2)+ \phi_{\rm FS}
\end{equation}
where $f_i$ is the local expression of the section $\sigma_i$ which defines the line $L_i$.

We define the section $u$ which equals $\sigma_2$ on $L_1$ and zero on $L_2$.
By Theorem \ref{otconique} we can construct an extension $U_\ep$ of $u$ for whose $L^2$ norm is bounded by
$$\int_{(\CC, 0)}\left|\frac{u}{z}\right|^{\frac{2}{1+\alpha}}e^{-\varphi_\ep}$$
which equals $\displaystyle \int_{(\CC, 0)}\frac{1}{\ep^2+ |z|^2}d\lambda$. In particular, we see that the bound \eqref{qi1} cannot be improved by replacing the
weight $e^{-\varphi_L}$ with $e^{-(1-\delta_0)\varphi_L}$, for any positive $\delta_0$.

\end{document}